\documentclass[10pt,reqno,a4paper]{amsart}

\usepackage{mdframed}
\usepackage{setspace}

\usepackage[latin1]{inputenc}
\usepackage{tikz}
\usetikzlibrary{shapes,arrows}
\usetikzlibrary{arrows.meta}
\usepackage{forest}
\usepackage{tikz-qtree}

\usetikzlibrary{arrows,shapes,positioning,shadows,trees}

\usepackage[linkcolor=red,colorlinks=true]{hyperref}

\usepackage{etoolbox}
\usepackage{amsmath}
\usepackage{enumerate}
\usepackage{amssymb}
\usepackage{amscd}
\usepackage{amsthm}
\usepackage{amsfonts}
\usepackage{graphicx}
\usepackage[all,cmtip]{xy}

\patchcmd{\subsection}{-.5em}{.5em}{}{}
\patchcmd{\subsubsection}{-.5em}{.5em}{}{}

\makeatother
\usepackage{hyperref}

\numberwithin{equation}{section}

\newcommand{\SL}{\operatorname{SL}}

\newcommand{\cC}{\mathcal{C}}

\newcommand{\cG}{\mathcal{G}}

\newcommand{\cL}{\mathcal{L}}

\newcommand{\cN}{\mathcal{N}}

\newcommand{\cS}{\mathcal{S}}
\newcommand{\cT}{\mathcal{T}}

\newcommand{\cX}{\mathcal{X}}
\newcommand{\cY}{\mathcal{Y}}
\newcommand{\cZ}{\mathcal{Z}}

\newcommand{\bN}{\mathbb{N}}

\newcommand{\bR}{\mathbb{R}}

\newcommand{\bZ}{\mathbb{Z}}

\newcommand{\ra}{\rightarrow}

\newcommand{\qand}{\quad \textrm{and} \quad}

\newcommand\subsetsim{\mathrel{%
\ooalign{\raise0.2ex\hbox{$\subset$}\cr\hidewidth\raise-0.8ex\hbox{\scalebox{0.9}{$\sim$}}\hidewidth\cr}}}
\newcommand{\eps}{\varepsilon}
\newcommand{\res}[3]{{}^{#1}\mathrm{res}^{#2}_{#3}}
\newcommand{\ind}[3]{{}^{#1}\mathrm{ind}^{#2}_{#3}}

\DeclareMathOperator{\pr}{pr}

\DeclareMathOperator{\supp}{supp}

\DeclareMathOperator{\Prob}{Prob}

\DeclareMathOperator{\gr}{graph}

\DeclareMathOperator{\Dens}{\underline{Dens}}

\DeclareMathOperator{\dom}{dom}
\DeclareMathOperator{\ran}{ran}

\definecolor{lichtgrijs}{gray}{0.95}
\mdfdefinestyle{mystyle}{ %
    backgroundcolor=lichtgrijs, %
    linewidth=0pt, %
    innertopmargin=10pt, %
    innerbottommargin=10pt,%
    nobreak=true
}

\theoremstyle{theorem}
\newtheorem{theorem}{Theorem}[section]
\newtheorem{corollary}[theorem]{Corollary}
\newtheorem{proposition}[theorem]{Proposition}
\newtheorem{lemma}[theorem]{Lemma}

\theoremstyle{definition}
\newtheorem{definition}[theorem]{Definition}

\newtheorem{remark}[theorem]{Remark}

\newtheorem{example}{Example}[section]

\tikzstyle{decision} = [diamond, draw, fill=blue!20, 
    text width=4.5em, text badly centered, node distance=3cm, inner sep=0pt]
\tikzstyle{block} = [rectangle, draw, fill=blue!20, 
    text width=5em, text centered, rounded corners, minimum height=2em]
\tikzstyle{line} = [draw, -latex']
\tikzstyle{cloud} = [draw, ellipse,fill=red!20, node distance=3cm,
    minimum height=2em]

\renewcommand\labelenumi{(\roman{enumi})}
\renewcommand\theenumi\labelenumi

\usepackage{changepage}
\usepackage{mathtools}
\usepackage{graphicx}
\usepackage{calc}

\usepackage[toc,page]{appendix}
\usepackage{mathrsfs}

\begin{document}
\bibliographystyle{plain} 

\title{Intersection spaces and multiple transverse recurrence}

\author[Michael Bj\"orklund]{Michael Bj\"orklund}
\address{}

\author[Tobias Hartnick]{Tobias Hartnick}
\address{}

\author[Yakov Karasik]{Yakov Karasik}
\address{}

\begin{abstract} We study multiple recurrence properties along separated cross sections for pmp actions of unimodular lcsc group on Polish spaces. We establish a multiple transverse recurrence theorem under the assumption that sufficiently large powers of the return time set are Delone sets. Typical examples of such situations arise from the theory of uniform approximate lattices.
\end{abstract}

\keywords{Transversal ergodic theory,  approximate lattices}

\subjclass[2010]{Primary: 37A15; Secondary: 37B20, 52C23}

\maketitle

\fontsize{10pt}{13pt}\selectfont

\section{Introduction}
Given a Borel measurable action of a locally compact second countable (lcsc) group $G$ on a Polish space $X$, we say that a point $x \in X$ is \emph{recurrent} if there exists a sequence $(g_n)$ in $G$  with $g_n \to \infty$ such that $g_n.x \to x$. The Poincar\'e recurrence theorem ensures that in the case of a  probability measure preserving (pmp) action, almost every point is recurrent. In this article we are interested in conditions which ensure a stronger recurrence property for generic points which we call \emph{multiple transverse recurrence}. 

More precisely, we are going to consider a pmp action of a lcsc group $G$ on a Polish space $X$ with invariant probability measure $\mu$.  If $Y \subset X$ is a sufficiently well-behaved cross-section,  then the measure $\mu$ gives rise to a transverse measure $\nu$ on $Y$ (see  \cite{Avni, KPV, Slutsky1}).  By a \emph{multiple transverse recurrence theorem} we shall mean a theorem which ensures that for every $r \geq 1$
and for $\nu^{\otimes r}$-almost all $(y_1, \dots, y_r) \in Y^r$ there exists a sequence $(g_n) \in G$ with $g_n \to \infty$ such that for all $j \in \{1, \dots, r\}$ we have both
 $g_n.y_j \to y_j$ (i.e.\ \emph{multiple recurrence}) and $g_n.y_j \in Y$ (i.e. \emph{transverse recurrence}).  Multiple transverse recurrence is a rare 
phenomenon,  which,  as we will see later,  only occurs under strong arithmetic assumptions on the the return time set to the cross-section.
 
 In this article we are going to establish multiple transverse recurrence theorem for a specific class of transverse systems whose return time sets are uniform approximate lattices in the sense of \cite{BjorklundHartnick1}. Before we go into the details of our setting, we discuss a specific example of a multiple transverse recurrence theorem, which motivated us to introduce the general setting considered below.

\subsection{A motivating example}\label{SecMotivation}
Let $G_o$ be a non-compact unimodular lcsc group with Haar measure $m_{G_o}$. We recall that the closed subsets of a lcsc group $G_o$ form a compact metrizable space $\mathcal C(G_o)$ under the Chabauty--Fell topology (see e.g. \cite{BjorklundHartnick1}) on which $G_o$ acts jointly continuously by $g.A := Ag^{-1}$. If $P_o$ is a discrete subset of $G_o$, then we denote by $\Omega_{P_o}$ the \emph{hull} of $P_o$, i.e. the orbit closure of $P_o$ in $\mathcal C(G_o)$, and set $\Omega^\times_{P_o} := \Omega_{P_o} \setminus\{\emptyset\}$. The subset $\mathcal T_{P_o} := \{Q \in \Omega_{P_o} \mid e \in P_o\}$ is then a \emph{cross section} for the $G_o$-action on $\Omega_{P_o}^\times$, i.e.\ every $G$-orbit in  $\Omega_{P_o}^\times$ intersects $\mathcal T_{P_o}$.

Motivated by problems in the theory of aperiodic order (see e.g. \cite{BjorklundHartnick1, BHP1, BHP2}) we would like to find conditions which ensure that for many elements $Q_1, \dots, Q_r$ of $\mathcal T_{P_o}$ the intersection $Q_1 \cap \dots \cap Q_r$ is large in a suitable sense. This can be achieved under two additional assumptions on the initial set $P_o$.

Firstly, we need to assume enough discreteness of $P_o$. Recall that a subset $P_o \subset G_o$ is called \emph{uniformly discrete} if its difference set $\Lambda := P_oP_o^{-1}$ does not accumulate at the identity; it is called \emph{relatively dense} if there exists a compact subset $K \subset G_o$ such that $G_o = KP_o$. In the sequel we assume that $P_o \subset G_o$ is a subset such that the difference set $\Lambda$ is relatively dense and $\Lambda^3$ is uniformly discrete. These assumptions imply that $\Lambda$ is actually a \emph{uniform approximate lattices} in the sense of \cite{BjorklundHartnick1}, i.e.\ a relatively dense and discrete approximate subgroup, and hence $\Lambda^n$ is actually uniformly discrete for all $n\in \bN$. 

Secondly, we will assume that $G_o$ is unimodular and that the set $\Omega^\times_{P_o}$ admits a $G_o$-invariant probability measure $\mu$. We are mostly interested in the case in which $G_o$ itself is uncountable. In this case, the cross section $\mathcal T_{P_o}$ is a $\mu$-nullset, hence it does not make sense to speak about $\mu$-generic points in $\mathcal T_{P_o}$. However, 
the theory of transverse measures ensures that there is a finite measure $\nu$ on $\cT_{P_o}$, called the \emph{transverse measure} of $\mu$, such that for every bounded non-negative Borel function $f: G \times \cT_{P_o} \to \bR_{\geq 0}$,
\[
\int_{\Omega_{P_o}}  \sum_{p \in P} f(p^{-1},p.P)\, d\mu(P) = \int_{G_o} \int_{\cT_{P_o}} f(g, Q)\, d\nu(Q) \, dm_{G_o}(g),
\]
and this provides us with a notion of genericity in $\cT_{P_o}$. There are plenty of examples of sets $P_o$ satisfying both assumptions, including all uniform model sets in the sense of \cite{BHP1}.

\begin{theorem}[Multiple transverse recurrence]\label{MTR1} For $\nu^{\otimes r}$-almost all $(Q_1, \dots, Q_r) \in \mathcal T_{P_o}^r$ there exist $g_n\in Q_1 \cap \dots \cap Q_r$ such that
\[
 g_n.Q_1 \ra Q_1, \quad \dots, \quad g_n.Q_r \ra Q_r \qand g_n \ra \infty.\]
\end{theorem}
\vspace{0.2cm}

Note that for every $n \in \bN$ and $j \in \{1,\dots, r\}$ we have $g_n.Q_j \in \mathcal T_{P_o}$, hence the recurrence to $(Q_1, \dots, Q_r)$ is in the transverse direction.
The theorem implies in particular that generically the intersection $Q_1 \cap \dots \cap Q_r$ is infinite, but one can actually say much more. Given a sequence $(G_t)$ of subsets of $G_o$ of positive Haar measures we define the \emph{lower $(G_t)$-density} of a locally finite subset $A \subset G_o$
 by
 \[
 \Dens_{(G_t)}(A) := \varliminf_{t \ra \infty} \frac{|A \cap G_t|}{m_G(G_t)}.
 \]

\begin{theorem}[Positive density of intersections]\label{PD1} For $\nu^{\otimes r}$-almost all $(Q_1, \dots, Q_r) \in \mathcal T_{P_o}^r$ and every convenient sequence $(G_t)$ in $G_o$ we have
\[
\Dens_{(G_t)}\big(Q_1 \cap \dots \cap Q_r\big) > 0.
\]
\end{theorem}
\vspace{0.2cm}
 
See Definition \ref{DefConvenient} below for the definition of a convenient sequence; for the purposes of this introduction it suffices to know that
such sequences exist in many amenable lcsc groups and all semisimple algebraic groups over local fields.

\subsection{A general setting}
We now describe a general abstract setting, in which Theorems \ref{MTR1} and \ref{PD1} hold true. For the rest of this introduction, $G_o$ denotes a unimodular lcsc group with a jointly continuous action on a Polish space $X$. A Borel subset $Z \subset X$ is called a \emph{cross section} for the $G_o$-action if $G_o.Z = X$; it is called \emph{separated} if the set $\Lambda(Z) := \{g \in G_o\mid g.Z \cap Z \neq \emptyset\}$ of \emph{return times} to $Z$ intersects some identity neighbourhood in $G_o$ only in $\{e\}$. For example, $\mathcal T_{P_o} \subset \Omega_{P_o}^\times$ is a cross section with return time set $\Lambda = P_oP_o^{-1}$, hence a separated cross section if $\Lambda$ is uniformly discrete. \\

From now on, $Z \subset X$ denotes a separated cross section and $\mu$ denotes a $G_o$-invariant probability measure on $X$. As in the motivating example there is then a finite measure $\nu$ on $Z$, called the \emph{transverse measure} of $\mu$, such that for every bounded non-negative Borel function $f: G \times Z \to \bR_{\geq 0}$,
\[
\int_X  \sum_{g \in Z_x} f(g^{-1},g.x)\, d\mu(x) = \int_{G_o} \int_Z f(g,z)\, d\nu(z) \, dm_{G_o}(g),
\]
where $Z_x :=\big\{ g \in G_o \mid  g.x \in Z \big\}$ denotes the set of \emph{return times} from $x$ to $Z$. 

\begin{definition}
We say that the action of $G_o$ on $(X, \mu)$ is 
\begin{itemize}
\item \emph{recurrent} if for $\mu$-almost every $x \in X$ there exists a sequence $(g_n)$ in $G_o$ such that $g_n \to \infty$ and $g_n.x \to x$.  \vspace{0.1cm}
\item \emph{$r$-fold recurrent} if for $\mu^{\otimes r}$-almost every $(x_1, \dots, x_r) \in X^r$  there exists a sequence $(g_n)$ in $G_o$ such that $g_n \to \infty$ and $g_n.x_j \to x_j$ for all $j \in \{1, \dots, r\}$.  \vspace{0.1cm}
\item \emph{transversally recurrent} if for $\nu$-almost every $z\in Z$ there exists a sequence $(g_n)$ in $G_o$ such that  $g_n \to \infty$, $g_n.z \in Z$ and $g_n.z \to z$ (i.e.\ the recurrence is along the transversal).  \vspace{0.1cm}
\item \emph{$r$-fold transversally recurrent} if for $\nu^{\otimes r}$-almost every $(z_1, \dots, z_r) \in Z^r$  there exists a sequence $(g_n)$ in $G_o$ such that $g_n \to \infty$, $g_n.z_j \in Z$ and $g_n.z_j \to z_j$ for all $j \in \{1, \dots, r\}$.
\end{itemize}
\end{definition}
\vspace{0.2cm}

While transversal reccurence follows from recurrence, multiple transverse recurrence does not follow from multiple recurrence. The reason is that $Z^r$ is not a transversal for the diagonal $G_o$-action on $X^r$, but only for the $G_o$-action on the (typically much smaller) \emph{intersection space}
\[
X^{[r]} = \{(x_1, \dots, x_r) \in X^r \mid \exists g \in G_o:\; g.x_j \in Z \text{ for all } j = 1, \dots, r\}.
\]
Note that in the case where $X = \Omega_{P_o}^\times$ and $Z = \cT_{P_o}$ this space is given by
\[
\{(Q_1, \dots, Q_r) \in \Omega_{P_o}^r\mid Q_1 \cap \dots \cap Q_r \neq \emptyset\},
\]
hence the name. In order to establish $r$-fold transverse recurrence we first construct a finite invariant measure on $X^{[r]}$ with transverse measure $\nu^{\otimes r}$ and then establish multiple recurrence for this measure. This can be carried out under suitable assumptions on the return time set $\Lambda$: 

\begin{theorem}[Finiteness of intersection measure]\label{Intro1} If $\Lambda$ is a uniform approximate lattice in $G_o$, then there exists a unique finite $G_o$-invariant measure $\mu^{[r]}$ on $X^{[r]}$ whose transverse measure on $Z^r$ is given by $\nu^{\otimes r}$.  Moreover,  if $\mu$ is $G_o$-ergodic,  then,  up to a multiplicative constant,  $\mu^{[r]}$ is an $r$-fold $G_o$-invariant self-joining of the measure $\mu$. 
\end{theorem}
\vspace{0.2cm}

Once the existence of a finite intersection measure is established, multiple transversal recurrence follows by standard ergodic theoretical methods (see Theorem \ref{Thm_Poincare} for the general statement):

\begin{theorem}[Multiple transverse recurrence]\label{MTR2} If $\Lambda$ is a uniform approximate lattice in $G_o$, then the action of $G_o$ on $X$ is $r$-fold transversally recurrent for every $r \in \bN$.
\end{theorem}
\vspace{0.2cm}

Another consequence of the existence of a finite intersection measures concerns intersections of return time sets (see Theorem \ref{Thm_Density} for the general statement).  Recall that for $z \in Z$ we denote by $Z_z \subset G_o$ the set of return times from $x$ to $Z$.  

\begin{theorem}[Positive density of intersections]\label{PD2} If $\Lambda$ is a uniform approximate lattice in $G_o$, then for every convenient sequence $(G_t)$ of subsets of $G$ we have
\[
\Dens_{(G_t)}(Z_{z_1} \cap \dots \cap Z_{z_r}) > 0,
\]
for $\nu^{\otimes r}$-almost all $(z_1, \dots, z_r) \in Z^r$.
\end{theorem}
\vspace{0.2cm}

Note that Theorems \ref{MTR1} and \ref{PD1} are special cases of Theorems \ref{MTR2} and \ref{PD2} respectively.

\subsection{Establishing finiteness of the intersection measure}

Before we discuss a few further generalizations of the theorems above we would like to comment on the methods behind the proof of Theorem \ref{Intro1}. It turns out that this theorem can be derived from a restriction in stages theorem for transverse measures with respect to certain semidirect products groups. To state the theorem we assume that $G = NL$ is the semidirect product of a normal subgroup $N$ and a subgroup $L$, where $N$ and $L$ are unimodular lcsc groups and the $L$-action on $N$ is Haar measure-preserving. We may then normalize Haar measures on $G$, $N$ and $L$ such that $m_G(W_N W_L) = m_N(W_N) m_L(W_L)$ for all Borel subsets $W_N \subset N$ and $W_L \subset L$. \\

We now consider a Borel action of a group $G$ on a Polish space $X$ with separated cross section $Z$. If we define $Y:=L.Z$, then $X \supset Y \supset Z$ and
\begin{itemize}
\item $G$ acts on $X$ with separated cross section $Z$;
\item $N$ acts on $X$ with cross section $Y$ (which need not be separated);
\item $L$ acts on $Y$ with separated cross section $Z$.
\end{itemize}
If $\mu$ is a finite $G$-invariant measure on $X$, then we denote the corresponding transverse measure (with respect to the $G$-action) on $Z$ by $\res{G}{X}{Z}(\mu)$. It turns out that this measure is finite and uniquely determines the original measure $\mu$. We also use similar notations for the other two groups. The technical heart of this article is then the proof of the following theorem: 

\begin{theorem}[Restriction in stages]\label{Stages} If the $N$-cross section $Y$ is separated, then for every finite $G$-invariant measure $\mu$ on $X$ the measure $\res{N}{X}{Y}(\mu)$ is finite and $L$-invariant and
\[
\res{G}{X}{Z} = \res{L}{Y}{Z} \circ \res{N}{X}{Y}.
\]
\end{theorem}
\vspace{0.2cm}

Let us demonstrate how Theorem \ref{Stages} implies a more general version of Theorem \ref{Intro1}. For this let $X_1, \dots, X_r$ be Polish spaces on which $G_o$ acts jointly continuously with corresponding separated cross sections $Z_1, \dots, Z_r$ and let $\mu_1, \dots, \mu_r$ be $G_o$-invariant probability measures on $X_1, \dots, X_r$ respectively. We then denote by  $\nu_1, \dots, \nu_r$ and $\Lambda_1, \dots, \Lambda_r$ the corresponding transverse measure, respectively return time sets and abbreviate
\[
G := G_o^r, \quad X := X_1 \times \dots \times X_r \qand Z := Z_1 \times \dots \times Z_r.
\]
Then $Z$ is a separated cross section for the action of $G$ on $X$. In order to apply Theorem \ref{Stages} we observe that the group $G$ admits semidirect product splittings of the form $G = N_k  L$ for any $k \in \{1, \dots, r\}$, where $N_k$ denotes the kernel of the $k$th coordinate projection and $L\cong G_o$ denotes the diagonal subgroup. In this situation, the intermediate transversal $Y := L.Z$ is given by the \emph{generalized intersection space}
\[
X^{[r]} = \Big\{ (x_1,\ldots,x_r) \in X_1 \times \dots \times X_r  \mid  \exists\, g \in G_o:\; g.x_1 \in Z_1, \dots, g.x_r \in Z_r\}.
\]
We say that the cross sections $Z_1, \dots, Z_r$ are \emph{commensurable} if $Y=X^{[r]}$ is a separated cross section for the $N_k$-action on $X$ for all $k \in \{1, \dots, r\}$ so that Theorem \ref{Stages} applies. In this case
\[
\mu^{[r]}_k := \res{N_k}{X}{Y}(\mu_1 \otimes \dots \otimes \mu_r)
\]
is a finite measure on $Y=X^{[r]}$ which is invariant under the action of the diagonal group $L \cong G_o$ and
\[
\res{L}{Y}{Z}(\mu^{[r]}_k) = \res{G}{X}{Z}(\mu_1 \otimes \dots \otimes \mu_r) =\nu_1 \otimes \dots \otimes \nu_r.
\]
Furthermore, $\mu^{[r]}_k$ projects to the given measure $\mu_k$ on $X_k$. Moreover, we have
\[
\res{L}{Y}{Z}(\mu^{[r]}_1) = \dots = \res{L}{Y}{Z}(\mu^{[r]}_r), \quad \text{and hence}\quad \mu^{[r]}_1 = \dots = \mu^{[r]}_r,
\]
i.e.\ the measure $\mu^{[r]} := \mu^{[r]}_k$ is actually independent of $k$. This then establishes the following general version of Theorem \ref{Intro1}:

\begin{theorem}[Finiteness of intersection measure, general form]\label{Joining} Assume that  $Z_1, \dots, Z_r$ are commensurable. Then $G_o$ acts diagonally on $X^{[r]}$ with separated cross section $Z := Z_1 \times \dots \times Z_r$ and there is a unique finite $G_o$-invariant measure $\mu^{[r]}$ on $X^{[r]}$ such that 
\[
\res{G_o}{X^{[r]}}{Z}(\mu^{[r]}) = \nu_1 \otimes \dots \otimes \nu_r.
\]
Moreover,  if $\mu_1,\ldots,\mu_r$ are $G_o$-ergodic,  then,  up to a multiplicative constant,  $\mu^{[r]}$ is a $G_o$-invariant joining of the measures $\mu_1, \dots, \mu_r$.
\end{theorem}

In order to apply the theorem, one needs to give criteria which ensure that the cross sections $Z_1, \dots, Z_r$ are commensurable. This is the point where we need discreteness and cocompactness assumptions on the return times: 

\begin{theorem}[Commensurability criterion]\label{ThmComm} The cross sections $Z_1, \dots, Z_r$ are commensurable provided the intersection $\Lambda := \Lambda_1 \cap \dots \cap \Lambda_r$ is relatively dense in $G_o$ and none of the product sets $\Lambda_i^3\Lambda_j^3$ accumulates at the identity.
\end{theorem}

Note that if $X_1 = \dots = X_r$ and $Z_1 = \dots = Z_r$ this condition is satisfied if and only if $\Lambda = \Lambda_1 = \dots = \Lambda_r$ is a uniform approximate lattice.

\subsection{Generalizations}

Just as Theorem \ref{Intro1} implies  Theorems \ref{MTR2} and \ref{PD2} (and hence in particular Theorems \ref{MTR1} and \ref{PD1}) we can deduce from Theorem \ref{Joining} the following theorems. 

\begin{theorem}[Multiple transverse recurrence, general form]\label{MTR3} If the cross sections $Z_1, \dots, Z_r$ are commensurable, then for $\nu_1 \otimes \dots \otimes \nu_r$-almost all $(z_1, \dots, z_r) \in Z_1 \times \dots \times Z_r$ there exists a sequence $(g_n)$ in $(Z_1)_{z_1} \cap \dots \cap (Z_r)_{z_r}$ such that
\[
g_nz_1 \to z_1, \quad \dots,\quad g_nz_r \to z_r \qand g_n \to \infty.
\]
\end{theorem}
\vspace{0.2cm}

\begin{theorem}[Positive density of intersections, general form]\label{PD3} If the cross sections $Z_1, \dots, Z_r$ are commensurable,  then for every convenient sequence $(G_r)$ of subsets of $G$,  we have 
\[
\Dens_{(G_t)}((Z_1)_{z_1} \cap \dots \cap (Z_r)_{z_r}) > 0,
\]
for $\nu_1 \otimes \dots \otimes \nu_r$-almost all $(z_1, \dots, z_r) \in Z_1 \times \dots \times Z_r$.
\end{theorem}
\vspace{0.2cm}

This allows us in particular to establish the following twisted versions of Theorems \ref{MTR1} and \ref{PD1}: 

\begin{corollary}[Twisted multiple transverse recurrence and twisted positive density]\label{IntroTwisted} Let $P_o$, $\mu$ and $\nu$ be as in Subsection \ref{SecMotivation} and let $\lambda_1, \dots, \lambda_r$ be contained in the subgroup of $G_o$ generated by $P_o$.
\vspace{0.1cm}
\begin{enumerate}[(i)]
\item For $\nu^{\otimes r}$-almost all $(Q_1, \dots, Q_r) \in \cT_{P_o}^r$ there exist $g_n \in \bigcap_{k=1}^r \lambda_k Q_k \lambda_k^{-1}$ such that
\[
 \lambda_1^{-1}g_n\lambda_1.Q_1 \ra Q_1, \quad \dots, \quad  \lambda_r^{-1}g_n\lambda_r.Q_r \ra Q_r \qand g_n \ra \infty.\]
 
\item For every convenient sequence $(G_t)$ in $G_o$,  we have
\[
\Dens_{(G_t)}\big(\lambda_1Q_1\lambda_1^{-1} \cap \dots \cap \lambda_r Q_r \lambda_r^{-1}\big) > 0,
\]
for $\nu^{\otimes r}$-almost all $(Q_1, \dots, Q_r) \in \cT_{P_o}^r$.
\end{enumerate}
\end{corollary}

We do not know whether separatedness of $Y$ can be replaced by a weaker condition in Theorem \ref{Stages}. If this was the case, then the assumptions in all of our theorems could be weakened accordingly.

\subsection{A general correspondence theorem for transverse measures}

It is obvious from the results listed above that transverse measures (of finite measures) play a central role in this article. In recent years there have been a number of excellent expositions of the theory of transverse measures, see in particular \cite{Avni, KPV, Slutsky1}, and in preparing this article we profited very much from these expositions. However, for the purposes of the present article we needed a version of transverse measure theory which applies to actions which are not necessarily essentially free and measures which are not necessarily finite (but only $\sigma$-finite). For the convenience of the reader we thus decided to include a self-contained treatment of the theory of transverse measures in this specific context. Let us briefly summarize this theory here: \\

We consider a Borel action $G \curvearrowright X$ of a unimodular lcsc group $G$ on a standard Borel space $X$ with separated cross section $Y$. We fix a Haar measure $m_G$ on $G$ throughout. The action of $G$ on $X$ defines an equivalence relation on $X$, the \emph{orbit relation} of $G \curvearrowright X$, and we denote by $R_{G,Y}$ (or $R_Y$ if $G$ is clear from context) the restriction of this orbit relation to $Y$; it is called the \emph{cross section equivalence relation}. We are going to construct a bijective correspondence between certain classes of $G$-invariant $\sigma$-finite measures on $X$ and certain classes of $R_{Y}$-invariant $\sigma$-finite measures on $Y$. (See Subsection \ref{SecRelationInvMeasure} below for the notion of invariance under an equivalence relation.) To obtain this perfect correspondence, we have to fix certain finiteness conditions on the measures in questions.  

\begin{definition} A \emph{Borel exhaustion} of $Y$ is an increasing sequence $\cY= (Y_n)$ in $\mathscr{B}_Y$ such that $Y = \bigcup Y_n$. A Borel measure $\nu$ on $Y$ is called \emph{$\cY$-finite} if $\nu(Y_n) < \infty$ for all $n$, and we denote by $M(Y, \cY)$ the space of all $\cY$-finite Borel measures on $Y$.
\end{definition}

We fix once and for all a Borel exhaustion $\cG = (K_n)$ of $G$ by compact sets, such that every compact set in $G$ is contained in some $K_n$. For every Borel exhaustion $\cY = (Y_n)$ of $Y$ we then obtain a Borel exhaustion $\cY_{\cG} := (K_n.Y_n)$ of $X$. We now denote by $M(X, \cY_{\cG})^G \subset M(X, \cY_{\cG})$ the subspace of $G$-invariant measures and by $M(Y, \cY)^{R_{G,Y}} \subset M(Y, \cY)$ the subspace of $R_{G,Y}$-invariant measures.

\begin{theorem}[Measure correspondence, general case]\label{GMCT} There exist mutually inverse bijections 
\[\res{G}{X}{Y}:M(X, \cY_{\cG})^G \to M(Y, \cY)^{R_{G,Y}}, \qand \ind GXY: M(Y, \cY)^{R_{G,Y}} \to M(X, \cY_{\cG})^G,\] 
such that the following hold for all paris $(\mu, \nu)$ with $\mu \in M(X, \cY_{\cG})^G$ and $\nu = \res{G}{X}{Y}(\mu)$:
\vspace{0.2cm}
\begin{enumerate}
\item For every bounded non-negative Borel function $f$ on $Y$ we have
\[
\nu(f) =   \int_X \sum_{g \in Y_x} f(g.x)\, d\mu(x).
\]
\item More generally, for every bounded non-negative Borel function $F$ on $G \times Y$ we have
\[
(m_G \otimes \nu)(F) = \int  \sum_{g \in Y_x} F(g^{-1},g.x) \, d\mu(x)
\]
\item If $C \subset G \times Y$ is a Borel subset such that $a: C \to X$, $(g,y) \mapsto g.y$ is injective, then
\[
\mu(a(C)) = (m_G \otimes \nu)(C).
\] 
\item If $B$ be a $G$-invariant Borel set in $X$, then
\[
\mu(B) = 0 \iff \nu(B \cap Y) = 0.
\]
\end{enumerate}
\end{theorem}
\vspace{0.2cm}

If $\mu \in M(X, \cZ_{\cG})^G$, then $\nu := \res{G}{X}{Y}(\mu)$ is called the \emph{transverse measure} of $\nu$ and $\mu$ is called the \emph{lifted measure} of $\nu$. An important special case of Theorem \ref{GMCT} concerns the case of Radon measures:

\begin{corollary}[Measure correspondence for Radon measures]\label{CorRadon} Let $G \curvearrowright X$ be a continuous action of a lcsc group $G$ on a lcsc space $X$ and let $Y \subset X$ be a closed separated cross section. Then $\res GXY$ restricts to a bijection between $G$-invariant Radon measures on $X$ and $R_{G,Y}$-invariant Radon measures on $Y$.
\end{corollary}


\subsection{Structure of the paper}

This paper is composed of three parts. \textbf{Part 1} (Section \ref{SecMeasureTheory} - Section \ref{SecTransverseMeasure}) develops transverse measure theory in our specific setting and establishes Theorem \ref{GMCT} and Corollary \ref{CorRadon}. \textbf{Part 2} (Sections \ref{SecTransversalRecurrence} and \ref{SecErgodicTransversal}) summarizes the basic ergodic theory of transverse measures. In particular, we establish Theorem \ref{MTR3} and Theorem \ref{PD3} in the case $r = 1$. Our main results are derived in \textbf{Part 3} (Section \ref{Sec:semidirect} - Section \ref{SecAUL}): In Section \ref{Sec:semidirect} we study transverse measure for semi-direct product groups and establish Theorem \ref{Stages}. We apply this in Section \ref{SecIntersectionSpaces} to the study of intersection spaces and establish Theorems \ref{Joining}  - \ref{PD3} (and hence also Theorems \ref{Intro1} - \ref{PD2} as special cases). In Section \ref{SecAUL} we specialize further to the case of return times which are uniform approximate lattices and establish Theorem \ref{MTR1}, Theorem \ref{PD1} and Corollary \ref{IntroTwisted}.

\setcounter{tocdepth}{1}
\tableofcontents


\section{Measure theory}\label{SecMeasureTheory}

This section summarizes our notation concerning measurable actions of lcsc groups on standard Borel spaces and the corresponding invariant measures. The first three subsections (Subsection \ref{SSGenNot}, Subsection \ref{SSStandardBorel} and Subsection \ref{SSBorelGSpaces}) contain standard material concerning general measurable spaces, standard Borel spaces and Borel-$G$-spaces respectively. In Subsection \ref{SecRelationInvMeasure} we discuss the notion of invariance of a set or a measure under a countable Borel equivalence relation, which will play a crucial role in the sequel.

\subsection{Measurable spaces}\label{SSGenNot}

Let $(X,\mathscr{B}_X)$ be a measurable space.  We refer to the elements of $\mathscr{B}_X$ as \emph{Borel sets}.  A \emph{Borel measure on $X$} is a 
$\sigma$-additive measure on $\mathscr{B}_X$ with values in $[0,\infty]$.  We denote by $M(X)$ the set of all Borel measures on $X$ (we suppress the dependence on the $\sigma$-algebra $\mathscr{B}_X$).  A \emph{Borel exhaustion} is an increasing sequence $\cX := \{ X_n : n \geq 1\}$ in $\mathscr{B}_X$ such that $X = \bigcup_n X_n$,  and 
a Borel measure $\mu$ on $X$ is \emph{$\cX$-finite} if $\mu(X_n) < \infty$ for all $n$.
The set of all $\cX$-finite Borel measures on $X$ is denoted by $M(X;\cX)$,  and a Borel measure is \emph{$\sigma$-finite} if it is $\cX$-finite for some Borel exhaustion $\cX$.  If $\mu(X) < \infty$,  then $\mu$ is of course $\cX$-finite for any Borel exhaustion 
$\cX$.  In this case,  we say that $\mu$ is a \emph{finite Borel measure},  and  if $\mu(X) = 1$,  we say that 
$\mu$ is a \emph{Borel probability measure}.  We denote by $M_\sigma(X)$,  $M_{\textrm{fin}}(X)$ and $\Prob(X)$ the spaces of $\sigma$-finite Borel measures,  finite Borel measures and Borel probability measures on $X$ respectively.  Clearly,
\[
\Prob(X) \subset M_{\textrm{fin}}(X) \subset M(X;\cX) \subset M_\sigma(X),
\]
for every Borel exhaustion $\cX$.  If $\mu$ is a $\sigma$-finite Borel measure on $X$,  
we say that a Borel set $B$ is \emph{$\mu$-null} if $\mu(B) = 0$ and \emph{$\mu$-conull} if $\mu(B^c) = 0$.  If $\mathscr{A} \subset \mathscr{B}_X$ is a sub-$\sigma$-algebra and $T \subset X$ is a Borel set,  we define the $\sigma$-algebra $\mathscr{A}|_T$ by $\mathscr{A}|_T = \big\{ B \cap T \,  : \,  B \in \mathscr{A} \big\}$.  We refer to $\mathscr{A}|_T$ as the \emph{restriction of $\mathscr{A}$ to $T$}.  We also 
define $\mathscr{B}_T := \mathscr{B}_X|_T$,  and refer to the elements in $\mathscr{B}_T$ as \emph{Borel sets in $T$}.

If $(X,\mathscr{B}_X)$ and $(Y,\mathscr{B}_Y)$ are measurable spaces,  then a map 
$\varphi : X \ra Y$ is \emph{Borel} if $\varphi^{-1}(\mathscr{B}_Y) \subset \mathscr{B}_X$.  If $\varphi$ is a 
bijection and both $\varphi$ and $\varphi^{-1}$ are Borel,  we say that $\varphi$ is a
\emph{Borel isomorphism}.  More generally,  if $A$ is a Borel set in $X$,  then a map $\varphi : A \ra Y$ is \emph{Borel} if $\varphi^{-1}(\mathscr{B}_Y) \subset \mathscr{B}_A$.  A map $f : A \ra [0,\infty]$ is Borel if $A_\infty := f^{-1}(\{\infty\}) \in \mathscr{B}_A$ and $f|_{A \setminus A_\infty}$ is Borel. 

\subsection{Standard Borel spaces}\label{SSStandardBorel}

We say that a measurable space $(X,\mathscr{B}_X)$ is a \emph{standard Borel space} if there exists a 
Polish (completely metrizable and separable) topology $\cT$ on $X$ such that $\mathscr{B}_X$ is the Borel $\sigma$-algebra generated by $\cT$.  In particular,  there
is a countable subset $\cS \subset \mathscr{B}_X$ which generates $\mathscr{B}_X$ such that 
\[
\bigcap_{x \in B \in \cS} B = \{x\},  \quad \textrm{for all $x \in X$}.
\]
We refer to $\cS$ as a (countable) \emph{separating family} for $X$.  \\

Let us collect some basic facts about standard Borel spaces. 

\begin{lemma}
\label{Lemma_Borel}
Let $(X,\mathscr{B}_X)$ and $(Y,\mathscr{B}_Y)$ be standard Borel spaces and let $\varphi : X \ra Y$ be a Borel map.  
\begin{itemize}
\item[$(i)$] If $Z \subset X$ is a Borel set,  then $(Z,\mathscr{B}_X|_Z)$ is a standard Borel space \cite[Corollary 13.4]{Kechris}.  \vspace{0.1cm}
\item[$(ii)$] If $\varphi $
is injective,  then $\varphi(X)$ is a Borel set in $Y$ and $\varphi : X \ra \varphi(X)$ is a Borel isomorphism \cite[Corollary 15.2]{Kechris} \vspace{0.1cm}
\item[$(iii)$] If $\varphi$ has countable fibers,  i.e.  $\varphi^{-1}(\{y\})$ is at
most countable for every $y \in Y$,  then $\varphi(X)$ is a Borel set in $Y$ and $\varphi$
admits a Borel right-inverse.  \cite[Corollary 18.10]{Kechris} 
\end{itemize}
\end{lemma}

We denote by $\mathscr{B}_{\Prob(X)}$ the smallest $\sigma$-algebra on $\Prob(X)$
with respect to which the maps $\mu \mapsto \mu(B)$,  where $B$ is a Borel set in $X$,
are measurable.  By \cite[Section 17.E]{Kechris},  if $(X,\mathscr{B}_X)$ is standard,  then so $(\Prob(X),\mathscr{B}_{\Prob(X)})$.

\subsection{Borel $G$-spaces}\label{SSBorelGSpaces}

Let $G$ be a locally compact and second countable (lcsc) group and let $(X,\mathscr{B}_X)$ be a standard Borel space.  We denote by 
$\mathscr{B}_G$ the Borel $\sigma$-algebra on $G$,  and note that both
$(G,\mathscr{B}_G)$ and $(G \times X, \mathscr{B}_G \otimes \mathscr{B}_X)$ are standard Borel spaces.  Suppose $a : G \times X \ra X,  \enskip (g,x) \mapsto g.x$ is an action of $G$ on $X$.  If $a$ is Borel,  we say that $(X,a)$ is a \emph{Borel $G$-space}.  If $H$ is a closed subgroup of $G$,  we denote by $a|_{H}$ the restriction of $a$ to $H \times X$,
and refer to the Borel $H$-space $(X,a|_{H})$ as the \emph{$H$-restriction of $(X,a)$} \\

Suppose $(X,a)$ is a Borel $G$-space.  A Borel set $B$ in $X$ is \emph{$G$-invariant} if $g.B = B$ for all $g \in G$.  A Borel measure $\mu$ on $X$ is \emph{$G$-invariant} if $\mu(g.B) = \mu(B)$ for every $g \in G$ and $B \in \mathscr{B}_X$,  and we say that $\mu$  is \emph{$G$-ergodic} if every $G$-invariant Borel set in $X$ is either $\mu$-null or $\mu$-conull.  We denote by $\Prob(X)^G,  M_{\textrm{fin}}(X)^G,  M(X;\cX)^{G}$ and $M_\sigma(X)^G$ the spaces of $G$-invariant Borel probability measures,  $G$-invariant finite measures,  $G$-invariant $\cX$-finite measures and $G$-invariant $\sigma$-finite measures on $X$ respectively.   \\

Suppose $(X,a)$ and $(T,b)$ are Borel $G$-spaces.  A Borel map $\pi : X \ra T$ such that $\pi(g.x) = g.\pi(x)$ for all $g \in G$ and $x \in X$ is called a \emph{$G$-map}.  Note that if $\pi$ is  a $G$-map and $\mu \in M_{\textrm{fin}}(X)^G$,  then 
$\pi_*\mu \in M_{\textrm{fin}}(T)^G$,  where $\pi_*\mu(C) = \mu(\pi^{-1}(C))$ for $C \in \mathscr{B}_T$.  We refer to $\pi_*\mu$ as the \emph{push-forward} of $\mu$ under $\pi$ (we only consider push-forwards of finite measures,  as push-forwards of $\sigma$-finite measures are not $\sigma$-finite in general).  We refer to $(T,\pi_*\mu)$ as the
\emph{$G$-factor} of $(X,\mu)$ induced by $\pi$.  If $H$ is a lcsc group and $p : G \ra H$ is a continuous homomorphism,  then every Borel $H$-space $(T,b_o)$ can be lifted to a Borel $G$-space $(T,b)$,  by setting $b(g,x) = b_o(p(g),x)$.  We then say that 
\emph{$G$ acts on $T$ via $p$}.  \\

If $(X_1,a_1),\ldots,(X_r,a_r)$ are Borel $G$-spaces,  let $a_\Delta$ denote the diagonal 
$G$-action on $X_1 \times \cdots \times X_r$,  i.e.  $a_{\Delta}(g,x) = (a_1(g,x_1),\ldots,a_r(g,x_r))$,  for $g \in G$ and $x = (x_1,\ldots,x_r)$.  Let $\pi_k$ denote the projection from $X_1 \times \cdots \times X_r$ onto $X_k$.  Note that $\pi_k$ is a $G$-map.
Suppose $\mu_k \in \Prob(X_k)^G$ for $k = 1,\ldots, r$.  A $G$-invariant Borel probability measure $\mu$ on $X_1 \times \cdots \times X_r$ is a \emph{joining} of $(X_1,\mu_1),\ldots, (X_r,\mu_r)$ if $(\pi_k)_*\mu = \mu_k$ for every $k$.

\subsection{Countable Borel equivalence relations}\label{SecRelationInvMeasure}

Let $(X,\mathscr{B}_X)$ be a standard Borel space.  A Borel subset 
$E \subset X \times X$ is a \emph{countable Borel equivalence relation} (cber) if $E$ is an equivalence relation on $X$ and for every $x \in X$,  the set $ \{ y \in X \,  : \,  (x,y) \in E \big\}$ is countable.  \\

Suppose $E$ is a countable Borel equivalence relation.  Let $A, B \subset X$ be Borel
sets and $\varphi : A \ra B$ a Borel surjection. Then $\varphi$ is a \emph{partial $E$-map} if $\gr(\varphi) \subset E \cap (A \times B)$.  Note that, according to our terminology, every partial $E$-map $\varphi: A \ra B$ is assumed to be \emph{onto} $B$. We thus refer to
$A$ as the \emph{domain} of $\varphi$ and to $B$ as the \emph{range} of $B$,  and we 
sometimes write $A = \dom(\varphi)$ and $B = \ran(\varphi)$.  We denote by $[[E]]$ the
set of all injective (hence bijective) partial Borel $E$-maps.  We then have the following notions of invariance under $E$:

\begin{definition} A Borel set $B \subset X$ is \emph{$E$-invariant}
if 
\[
\varphi^{-1}(B \cap \ran(\varphi)) = B \cap \dom(\varphi),  \quad \textrm{for all $\varphi \in [[E]]$}.
\]
A Borel measure $\mu$ on $X$ is \emph{$E$-invariant} if $\mu(\dom(\varphi)) = \mu(\ran(\varphi))$,  for all $\varphi \in [[E]]$,  and an $E$-invariant Borel measure $\mu$ on $X$ is \emph{$E$-ergodic} if every $E$-invariant Borel set in $X$ is either $\mu$-null or $\mu$-conull.
\end{definition}

We denote by $\Prob(X)^E,  M_{\textrm{fin}}(X)^E,  M(X;\cX)^{E}$ and $M_\sigma(X)^E$ the spaces of $E$-invariant Borel probability measures,  $E$-invariant finite measures,  $E$-invariant $\cX$-finite measures and $E$-invariant $\sigma$-finite measures on $X$ respectively. In terms of Borel functions, $E$-invariance can be characterized as follows:
\begin{lemma}
\label{Lemma_Integral}
Suppose $\mu$ is an $E$-invariant Borel measure and $f : X \ra [0,\infty]$ is Borel.  Then,  
\[
\int_{\dom(\varphi)} f \circ \varphi \,  d\mu = \int_{\ran(\varphi)} f \,  d\mu,   \quad \textrm{for every $\varphi \in [[E]]$}.
\]
\end{lemma}
\vspace{0.1cm}
\begin{proof}
If $B_\infty := f^{-1}(\{\infty\}) \cap \ran(\varphi)$ has positive $\mu$-measure,  then so has $\varphi^{-1}(B_\infty) \cap \dom(\varphi)$.  Hence,  
\[
\int_{\dom(\varphi)} f \circ \varphi \,  d\mu = \infty \qand \int_{\ran(\varphi)} f \,  d\mu = \infty,
\]
and thus the identity in the lemma trivially holds.  Let us from now on assume that $\mu(B_\infty) = 0$. 
Suppose $f = \chi_B$ for some Borel set $B \subset X$.  If $\varphi \in [[E]]$,  we
define $\varphi_B := \varphi_{\varphi^{-1}(B) \cap \dom(\varphi)}$,  and note 
that $\varphi_B \in [[E]]$ with $\dom(\varphi_B) = \varphi^{-1}(B) \cap \dom(\varphi)$
and $\ran(\varphi_B) = B \cap \ran(\varphi)$. Since $\mu$ is $E$-invariant,  we have $\mu(\dom(\varphi_B)) = \mu(\ran(\varphi_B))$,  and thus 
\[
\int_{\dom(\varphi)} f \circ \varphi \,  d\mu = \mu(\dom(\varphi_B)) = \mu(\ran(\varphi_B)) = \int_{\ran(\varphi)} f \,  d\mu.
\]
We conclude that the lemma holds for every simple function.  A standard approximation argument finishes the proof. 
\end{proof}

\section{Separated cross-sections and lifted measures}

In this section we discuss cross-sections of Borel actions. We then explain how, under certain separability assumptions on the cross sections, measures on the cross section can be extended to measures on the whole space. Our main result is Theorem \ref{LiftingInvariantMeasure} which states that $\sigma$-finite measures on the transversal which are invariant under a certain equivalence relation can be canonically extended to $G$-invariant measures on the whole space, which are subject to a certain finiteness property.

\subsection{Separated cross sections}
Let $G$ be a lcsc group and $(X,a)$ a Borel $G$-space.  If $Y \subset X$ is a Borel set
and $x \in X$,  we define the \emph{set of hitting times} $Y_x$ by
\[
Y_x := \big\{ g \in G \,  : \,  g.x \in Y \big\} \subset G,
\]
and the \emph{set of return times} $\Lambda_a(Y)$ by
\[
\Lambda_a(Y) := \big\{ g \in G \,  : \,  Y \cap g^{-1}.Y \neq \emptyset \big\}= \bigcup_{y \in Y} Y_y.  
\]
Let $U$ be an open identity neighbourhood in $G$.  A Borel set $Y \subset X$ is 
\emph{$U$-separated} if $Y_y \cap U = \{e\}$ for all $y \in Y$ (or equivalently,  if $\Lambda_a(Y) \cap U = \{e\}$),  and a \emph{cross-section} if $G.Y = X$ (or equivalently,  if the restriction of the action map $a$ to
$G \times Y$ is surjective). If $Y \subset X$ is a $U$-separated cross section for some open identity neighbourhood $U$ in $G$, then we call $Y$ a \emph{separated cross-section} for short.
In this case we say that a Borel set $C \subset G \times Y$ is \emph{injective} if the restriction of $a$ to $C$ is injective. The following lemma ensures that for every separated cross section $Y \subset X$  every Borel set $B \subset X$ can be covered by the images of countably many injective sets in $G \times Y$.
\begin{lemma}
\label{Lemma_BasicProp}
Suppose $Y \subset X$ is a $U$-separated cross section.  
\vspace{0.1cm}
\begin{itemize}
\item[$(i)$] If $V \subset G$ is a Borel set such that $V^{-1}V \subset U$,  then 
$V \times Y$ is an injective Borel set.  In particular,  $|Y_x \cap V^{-1}| \leq 1$ for all $x \in X$.\vspace{0.1cm}
\item[$(ii)$] For every compact set $K \subset G$,  
\begin{equation}
\label{DefMK}
M_K := \sup_{x \in X} |Y_x \cap K| < \infty.
\end{equation}
\item[$(iii)$] For every $x \in X$,  the set $Y_x$ is countable and $Y_x Y_x^{-1} \subset \Lambda_a(Y)$.  In particular,  $Y_x$ is a $U$-uniformly discrete (and hence closed) subset of $G$. \vspace{0.1cm}. 
\item[$(iv)$] The map $a : G \times Y \ra X$ has countable fibers.  In particular,  
if $C \subset G\times Y$ is a Borel set,  then $a(C)$ is a Borel set in $X$ (by Lemma \ref{Lemma_Borel} (iii)).  \vspace{0.1cm}
\item[$(v)$] For every Borel set $B \subset X$,  there are injective Borel sets 
$\widetilde{B}_1,  \widetilde{B}_2,\ldots$ in $G \times Y$ such that 
\[
B = \bigsqcup_k \, a(\widetilde{B}_k).  
\]
\end{itemize}
\end{lemma}

\begin{proof}
(i) Suppose $(v_1,y_1),  (v_2,y_2) \in V \times Y$ such that $v_1.y_1 = v_2.y_2$.  
Since $Y$ is $U$-separated,  
\[
v_1^{-1}v_2 \in \Lambda_a(Y) \cap V^{-1}V \subset \Lambda_a(Y) \cap U = \{e\},
\]
and thus $v_1 = v_2$,  whence $y_1 = y_2$.  (ii) Let $K \subset G$ be a compact set,  and pick an open cover 
$V_1,\ldots,V_p$ of $K^{-1}$ such that $V_k^{-1}V_k \subset U$ for all $k$.  By (i), 
$|Y_x \cap V_k^{-1}| \leq 1$ for all $x$ and $k$,  and thus
\[
|Y_x \cap K| \leq \sum_{k=1}^p |Y_x \cap V_k^{-1}| \leq p,  \quad \textrm{for all $x$}.
\]
(iii) Countability of $Y_x$ is immediate from (ii).  To prove the inclusion,  pick $g_1,  g_2 \in Y_x$, 
and note that
\[
y := g_2.x \in Y \qand g_1.x = (g_1 g_2^{-1}).y \in Y,
\]
whence $g_1g_2^{-1} \in \Lambda_a(Y)$.  Since $g_1$ and $g_2$ are arbitrary,  $Y_x Y_x^{-1} \subset \Lambda_a(Y)$.  (iv) Fix $x \in X$ and an element $(g,y) \in a^{-1}(\{x\})$.  Then,  for any $(g',y') \in a^{-1}(\{x\})$,  we have $g'.y' = g.y = x$,  and thus $(g')^{-1} \in Y_x$ and $y' \in Y_x.x$. By (iii),  $Y_x$ is a countable set,  so there are only countably many choices for $g'$ and $y'$.  (v) Since $G$ is separable,  we can find a countable open cover 
\[
G = \bigcup_k V_k, \enskip \textrm{such that $V_k^{-1}V_k \subset U$ for all $k$}.
\]
Let $B'_k := B \cap V_k.Y$,  and note that  $B = \bigcup_k B_k'$.  Define 
\[
B_1 := B_1' \qand B_{k+1} = B'_{k+1} \setminus \Big( \bigcup_{j=1}^k B_j\Big),
\]
for all $k$.  Then,  $B_1,B_2,\ldots$ are disjoint Borel sets,  $B = \bigsqcup_k B_k$ and $B_k \subset B'_k \subset V_k.Y$ for all $k$.  Set $\widetilde{B}_k = a^{-1}(B_k) \cap (V_k \times Y)$.  By (i),  $V_k \times Y$ is an
injective set,  and thus $\widetilde{B}_k$ is an injective set as well,  and $a(\widetilde{B}_k) = a(a^{-1}(B_k) \cap (V_k \times Y)) = B_k$,  for all $k$.
\end{proof}

\subsection{Transverse triples and their Chabauty--Fell maps}

Throughout this subsection, $G$ denotes a lcsc group. We are interested in triples $(X,a,Y)$ where $(X,a)$ is a Borel $G$-space and $Y \subset X$ is a separated cross section. We refer to such a triple $(X,a,Y)$ as a \emph{transverse triple} over $G$ (or more specifically a \emph{$U$-transverse triple} if $Y$ is $U$-separated). In this subsection we discuss an important class of examples of transverse triples which arise from certain discrete subsets of $G$. In fact, we are going to see that this class of examples is ``universal'' in a suitable sense.

We denote by $\cC(G)$ the space of all closed subsets of $G$.  The \emph{Chabauty-Fell topology} is the topology on $\cC(G)$ generated by
open sets of the form
\[
W_K := \big\{ P \in \cC(G) \, : \,  P \cap K = \emptyset \big\}
\qand
W^V := \big\{ P \in \cC(G) \,  : \,  P \cap V \neq \emptyset \big\},
\]
where $K$ and $V$ range over all compact subsets and over all open subsets of $G$ respectively (cf.\ \cite{BjorklundHartnick1, BHP1}). It is not hard to see that with respect to this topology,  $\cC(G)$ is a compact and second countable Hausdorff space,  and a sequence $(P_n)$ converges to a point $P$ in $\cC(G)$ if and only if 
\vspace{0.1cm}
\begin{itemize}
\item[(i)] for every $p \in P$,  there exists a sequence $(p_n)$ in $G$ such that $p_n \in P_n$ for all $n$ and $p_n \ra p$ as $n \ra \infty$.  \vspace{0.1cm}
\item[(ii)] if $(n_k)$ is a sub-sequence and $p_k \in P_{n_k}$ such that $p_k \ra p$
as $k \ra \infty$,  then $p \in P$.
\end{itemize}
Furthermore,  the action 
\[
a_r : G \times \cC(G) \ra \cC(G), \enskip (g,P) \mapsto Pg^{-1}
\]
is jointly continuous.  If $U$ is an identity neighbourhood in $G$,  then a subset $P_o \in \cC(G)$ is \emph{$U$-uniformly discrete} if $P_o P_o^{-1} \cap U = \{e\}$.  We denote by $\cC^U(G) \subset \cC(G)$ the subset of $U$-uniformly discrete subsets of $G$. It is easy to see that $\cC^U(G)$ is a closed (hence compact) $G$-invariant subset of $\cC(G)$. The subset $\cT := \big\{ P \in \cC(G) \,  : \,  e \in P \big\} \subset \cC(G)$ is closed,  and thus compact. By definition, $\cT$ is a cross section for the $G$-action on $\cC(G) \setminus\{\emptyset\}$.

\begin{example}\label{CFExamples}
Assume that $P_o \subset G$ is a non-empty closed subset. We then define the \emph{hull} of $P_o$ as the orbit closure $\Omega_{P_o} := \overline{G.P_o}$ and set $\Omega_{P_o}^\times := \Omega_{P_o} \setminus\{\emptyset\}$. Then $\cT_{P_o} := \cT \cap \Omega_{P_o}$ is a cross section for the $a_r$-action of $G$ on $\Omega_{P_o}^\times$, though it will in general not be separated.
\end{example}

\begin{proposition}\label{UDExamples} The cross section $\cT_{P_o}$ is $U$-separated if and only if $P_o$ is $U$-uniformly discrete, hence $(X,a,Y) := (\Omega_{P_o}^\times, a_r, \cT_{P_o})$ is a transverse triple if and only if $P_o$ is uniformly discrete.
\end{proposition}
\begin{proof} We may assume without loss of generality that $e \in P_o$. For $Q \in Y = \mathcal T_{P_0}$ we have $Y_Q = Q \subset QQ^{-1}$ and by \cite[Lemma 3.15]{BHP2} we have $PP^{-1} \subset \overline{P_oP_o^{-1}}$ for all $P \in \Omega_{P_o}$. Thus
\[
\Lambda_{a_r}(Y) = \bigcup_{Q  \in \cT_{P_o}} Q \subset \bigcup_{Q  \in \cT_{P_o}} QQ^{-1} \subset  \overline{P_oP_o^{-1}},
\]
hence if $P_o$ is $U$-uniformly discrete, then $Y$ is $U$-separated. Conversely, if $Y$ is $U$-separated, then $P_o$ is $U$-uniformly discrete by Lemma \ref{Lemma_BasicProp} (iii).
\end{proof}
Proposition \ref{UDExamples} provides plenty of transverse triples $(X, a, Y)$, where $X$ is a subset of $\cC^U(G)$ and $Y = X \cap \cT$. Our next goal is to show that every transverse triple over $G$ admits a Borel-$G$-factor, which is of this form. For the proof we are going to use the fact that since every open set in $G$ is $\sigma$-compact,  the Borel $\sigma$-algebra $\mathscr B_{\cC(G)}$ is generated by the sets $W_K$ with $K$ relatively compact.

\begin{lemma}\label{CFEx} For every $U$-transverse triple $(X,a, Y)$ the map 
\[
\pi : X \ra \cC^U(G),  \enskip x \mapsto Y_x.
\]
is a well-defined Borel $G$-map with $\pi(Y) \subset \cT$ and $\pi^{-1}(\cT) = Y$.
\end{lemma}

\begin{proof} It follows from Lemma \ref{Lemma_BasicProp} (iii) that $Y_x \in \cC^U(G)$ for every $x \in X$, hence $\pi$ is well-defined. Note that
\[
\pi(g.x) = Y_{g.x} = \big\{ t \in G \, : \,  tg.x \in Y \big\} = Y_x g^{-1} = g.\pi(x),
\]
for all $g \in G$ and $x \in X$,  so $\pi$ is $G$-equivariant.  Furthermore,  
\[
\pi^{-1}(\cT) = \{ x \in X \,  : \,  e \in Y_x \} = Y.
\]
To prove that $\pi$ is Borel,  it suffices to show that $\pi^{-1}(W_K)$ belongs to $\mathscr{B}_X$.  Note that
\[
\pi^{-1}(W_K) = \big\{ x \in X \, : \,  Y_x \cap K = \emptyset \big\} = (K^{-1}.Y)^c.
\]
By Lemma \ref{Lemma_BasicProp} (iv),  $K^{-1}.Y$ is a Borel set in $X$,  so we are done.
\end{proof}
\begin{definition}
The map $\pi: X \to \cC^U(G)$ is referred to as the \emph{canonical Chabauty-Fell $G$-map}.
\end{definition}

We will be particularly interested in  transverse triples $(X,a,Y)$ for which $X$ admits a $G$-invariant ergodic probability measure $\mu$.  In this case,  if $\pi$ denotes the canonical Chabauty-Fell $G$-map of $(X,a, Y)$,  then $\eta := \pi_*\mu$ is a $G$-ergodic Borel probability measure on $\cC^U(G)$ with $\eta(\{\emptyset\}) = 0$.  Note that $\supp(\eta)$ is a closed,  hence compact,  $G$-invariant subset of $\cC^U(G)$. 
Since the latter space is second countable and $\eta$ is ergodic,  there exists an element $P_o \in \supp(\eta)$ whose $G$-orbit is dense in the support.  In particular,  
$\Omega_{P_o} \subset \supp(\eta)$.  Since $\eta(\{\emptyset\}) = 0$,  we conclude that $X' := \pi^{-1}(\Omega_{P_o}^{\times})$ is a $G$-invariant $\mu$-conull Borel subset of $X$,  and we thus have a Borel $G$-map
\[
\pi: (X', \mu|_{X'}) \to (\Omega^\times_{P_o}, \eta) \quad \text{such that} \quad Y \cap X' = \pi^{-1}(\mathcal T_{P_o}).
\]
Note that $\Lambda_a(Y) \subset  \Lambda_{a_r}(\mathcal T_{P_o})$.  Here is a typical class of examples:





\begin{example}[Cut-and-project sets]
Let $H$ be a lcsc group and suppose $\Gamma < G \times H$ is a lattice whose projection to $H$ is dense.  Then $X := (G \times H)/\Gamma$,  equipped with quotient
Borel structure,  is a Borel $G$-space,  with a unique $G$-invariant (and $G$-ergodic) Borel probability measure $\mu$ (see \cite[Lemma 5.7]{BHP1}).  

Fix a pre-compact subset $W \subset H$ with non-empty interior,  and set $Y := (\{e\} \times W)\Gamma$.  It is not hard to check that $Y$ is a separated cross-section.  We thus obtain an associated Chabauty-Fell map $\pi:  (G \times H)/\Gamma \to \mathcal C^U(G)$.  The elements of $\pi(X)$ are called \emph{cut-and-project sets} in the literature.  Our arguments above show that $\mu$-almost every $x \in X$,  the $G$-hull $\Omega_{\pi(x)}^{\times}$ supports a $G$-invariant Borel probability measure $\eta_x$(which might depend on the point $x$).  Under some additional mild assumptions on $W$ \cite[Theorem 1.1]{BHP1} tells us that there is in fact a unique $G$-invariant (and thus $G$-ergodic) Borel probability measure $\eta$ on $\pi(X)$,  not supported on the empty set.  In particular,  in this case,  $\eta_x = \eta$ for all $x$.  

\end{example}

\subsection{Borel $Y$-sections and lifted measures}
Let $G$ be a lcsc group with left-Haar measure $m_G$ and let $(X,a, Y)$ be a $U$-transverse triple over $G$ for some open identity neighbourhood $U$ in $G$. By Lemma \ref{Lemma_Borel} (iii) and Lemma \ref{Lemma_BasicProp} (iv). the map \[a|_{G \times Y}: G \times Y \to X\] is surjective with countable fibers, hence admits a Borel section $b: X \to G \times Y$. Any such section is necessarily of the form
\begin{equation}\label{def_bbeta}
b(x) = b_\beta(x) :=  (\beta(x)^{-1},\beta(x).x),
\end{equation}
for some Borel map $\beta: X \to G$ such that $\beta(x) \in Y_x$ for all $x \in X$. In the sequel we are going to refer to such a map $\beta$ as a \emph{Borel $Y$-section}; such sections then always exist by the aforementioned lemmas. 

\begin{definition}\label{def_nubeta} If  $\beta: X \to G$ is a Borel $Y$-section and $\nu$ is a Borel measure on $Y$, then the Borel measure $\nu_\beta$ on $X$ given by
\[
\nu_\beta(B) = m_G \otimes \nu(b_\beta(B)),  \quad \textrm{for $B \in \mathscr{B}_X$}
\]
is called the \emph{$\beta$-lifted measure} of $\nu$.
\end{definition}

Our next goal is characterize the finiteness properties of lifted measures of $\cY$-finite Borel measures on $Y$, where $\cY = \{ Y_n : n \geq 1 \big\}$ is a given Borel exhaustion of $Y$.
By \cite[Theorem 2.A.10]{CornulierHarpe}, there is a fundamental exhaustion $\cG = \{ K_n : n \geq 1\}$ of $G$ by compact sets,  i.e.  every $K_n$ is a compact set,  and every compact set in $G$ is eventually contained in some $K_n$. We fix such an exhaustion once and for all and define $\cY_{\cG} := \{ K_n.Y_n : n \geq 1 \}$.  We refer to $\cY_{\cG}$ as the \emph{$\cG$-suspension of $\cY$}.

\begin{lemma}
\label{Lemma_YGsuspension}
$\cY_{\cG}$ is a Borel exhaustion of $X$,  and a Borel measure $\mu$ on $X$ is 
$\cY_{\cG}$-finite if and only if $\mu(K.Y_n) < \infty$ for every $n$ and for every compact set $K$ in $G$. In particular, this notion does not depend on the choice of fundamental exhaustion $\cG$.
\end{lemma}

\begin{proof}
The only non-trivial part of the lemma is to show that $K.Y_n$ is a Borel set for every
compact set $K$ in $G$.  However,  $K \times Y_n$ is a Borel set in $G \times Y$ and 
by Lemma \ref{Lemma_BasicProp} (iv),  $a|_{G \times Y}$ is Borel with countable fibers,  so
$a(K \times Y_n) = K.Y_n$ is a Borel set by Lemma \ref{Lemma_Borel} (iii).
\end{proof}
\begin{lemma}
\label{Lemma_adaptedBorelYsection}
Let $\{ K_n\}$ be a fundamental exhaustion of 
$G$ by compact sets,  and $\{ Y_n \}$ a Borel exhaustion of $Y$.  Then there is a Borel section $\beta$ such that $b_\beta(K_n.Y_n) \subset K_n \times Y_n$ for all $n$.
\end{lemma}

\begin{proof}
By Lemma \ref{Lemma_Borel} (iii) and Lemma \ref{Lemma_BasicProp} (iv),  the 
map $a_n := a|_{K_n \times Y_n}$ has a Borel right-inverse 
\[
b_n : K_n.Y_n \ra K_n \times Y_n,  \enskip x \mapsto (b_{n,1}(x),b_{n,2}(x)). 
\]
Note that both $b_{n,1}$ and $b_{n,2}$ are Borel maps,  and that they satisfy $b_{n,1}(x).b_{n,2}(x) = x$ for all $x \in K_n.Y_n$.  We set $\beta_n(x) := b_{n,1}(x)^{-1}$,  so that
\[
b_n(x) = (\beta_n(x)^{-1},\beta_n(x).x),  \quad \textrm{for all $x \in K_n.Y_n$}.
\]
In particular,  $\beta_n(x) \in Y_x \cap K_n^{-1}$ for all $x \in K_n.Y_n$.  Let $X_1 := K_1.Y_1$,  and
define inductively (the possibly empty Borel sets) $X_2,  X_3,\ldots$ by 
$X_{n+1} := K_{n+1}.Y_{n+1} \setminus \bigcup_{m=1}^n X_m$.  Since $\{K_n.Y_n\}$
is a Borel exhaustion of $X$,  we see that $X_1,X_2,\ldots$ is a Borel partition of $X$.
We now define $\beta(x) = \beta_n(x)$ for $x \in X_n$,  which is clearly a Borel $Y$-section with the property that 
\[
\beta(K_n.Y_n) = \bigcup_{m=1}^n \beta(X_m) \subset \bigcup_{m=1}^n K_m^{-1} = K_n^{-1},
\]
since $\{K_n\}$ is increasing.
\end{proof}
In the sequel, when constructing lifted measures, we will usually use $Y$-sections as in Lemma \ref{Lemma_adaptedBorelYsection}. In this case, lifts of $\sigma$-finite measures will again be $\sigma$-finite. More precisely:
\begin{corollary}\label{cor_YGFinite} Let $\beta: X \to G$ be a Borel $Y$-section such that $b_\beta(K_n.Y_n) \subset K_n \times Y_n$ for all $n$. Then for all
$\nu \in M(Y, \cY)$ we have $\nu_\beta \in M(X, \cY_{\cG})$.
\end{corollary}
\begin{proof} This is immediate from the fact that $\nu_\beta(K_n.Y_n) \leq (m_G \otimes \nu)(K_n \times Y_n)$ for all $n$.
\end{proof}

\subsection{Cross-section equivalence relations}

Let $G$ be a lcsc group and let $(X,a, Y)$ be a $U$-transverse triple over $G$ for some open identity neighbourhood $U \subset G$. We
define the \emph{cross-section equivalence relation} $R_Y \subset Y \times Y$ by
\[
R_Y := \big\{ (y,y') \in Y \times Y \,  : \,  y \in G.y' \big\}.
\]
By definition, $R_Y$ is just the restriction of the orbit relation of the $G$-action on $X$ to the cross-section $Y$, hence in particular a (set-theoretic) equivalence relation on $Y$. Moreover, a straightforward application of Lemma \ref{Lemma_BasicProp} yields:

\begin{lemma}
$R_Y$ is a countable Borel equivalence relation on $Y$. 
\end{lemma}

\begin{proof}
First note that if $y' \in Y$,  then $y \in G.y$ if and only if $y \in Y_{y'}.y'$.  By Lemma \ref{Lemma_BasicProp} (iii),  $Y_{y'}$ is countable,  whence the set $\{ y \in Y \,  : \,  (y,y') \in R_Y\}$ is countable as well.  To see why $R_Y$ is a Borel set in $Y \times Y$,  
consider the Borel map $\widetilde{a} : G \times Y \ra Y \times X,  \enskip (g,y) \mapsto (y,g.y)$.  Then
$C = \widetilde{a}^{-1}(Y \times Y)$ is a Borel set in $G \times Y$ and $R_Y = \widetilde{a}(C)$.  It follows from Lemma \ref{Lemma_BasicProp} (iv) that $\widetilde{a}$ has countable fibers,  so $R_Y$ is a Borel set in $Y \times Y$ by Lemma \ref{Lemma_Borel} (iii).
\end{proof}
Our next goal is to study sets and measures which are invariant under the cross-section equivalence relation. For this we need to describe the local structure of partial $R_Y$-maps. The following lemma is based on the proof of \cite[Lemma 2.5]{Avni}. 

\begin{lemma}
\label{Lemma_partialstructure}
Let $\varphi$ be a partial $R_Y$-map and let $V$ be a symmetric identity neighbourhood in $G$ such that $V^4 \subset U$.  
Then there exist
\vspace{0.1cm}
\begin{itemize}
\item a Borel partition $\dom(\varphi) = \bigsqcup_k A_k$ such that $\varphi|_{A_k}$ is injective for every $k$.  \vspace{0.1cm} 
\item identity neighbourhoods $W_k \subset V$ and $\lambda_k \in \Lambda_a(Y)$ such that $\, \lambda_k^{-1} W_k \lambda_k \subset V$ for all $k$. \vspace{0.1cm}
\item Borel maps $\rho_k : A_k \ra W_k$ such that
\[
\varphi(y) = \rho_k(y)\lambda_k.y,  \enskip \textrm{for all $y \in A_k$}.
\]
\end{itemize}
\vspace{0.1cm}
In particular,  $a(C_k) = W_k\lambda_k.A_k$, where 
\vspace{0.1cm}
\[
C_k = \big\{ (w\rho_k(y)^{-1},\varphi(y)) \, : \, w \in W_k, \enskip y \in A_k \big\} \subset W_k^2 \times Y.
\]
\end{lemma}

\begin{proof}
Set $A := \dom(\varphi)$.  For every $y \in A$,  pick $\lambda_y \in \Lambda_a(Y)$ such that $\varphi(y) = \lambda_y.y$,  and a symmetric open identity neighbourhood $W_y \subset V$ such that $\lambda_y^{-1}W_y \lambda_y \subset V$.  Since $G$ is second countable,  and thus Lindel\"of,  we can find $y_1,y_2,\ldots \in A$ such that
\[
W := \bigcup_{y \in Y} W_y \lambda_y = \bigcup_{k} W_{y_k} \lambda_{y_k}.
\]
Set $W_k := W_{y_k}$ and $\lambda_k := \lambda_{y_k}$.  Let $A_k' := \big\{ y \in A \, : \, \varphi(y) \in W_k\lambda_k.y \big\}$,  and note that 
\[
\bigcup_k A_k' = \big\{ y \in A \, : \, \varphi(y) \in W.y \big\} = A.
\]
Define $A_1 := A_1'$ and $A_{k+1} = A_{k+1}' \setminus \Big( \bigcup_{j=1}^{k} A_j\Big)$,
for all $k$.  Then $A_1,A_2,\ldots$ are disjoint Borel sets with union $A$ (we can discard empty sets from this collection).  Let us now 
construct the maps $\rho_k$.  We define the Borel set
\[
\widetilde{A}_k = \big\{ (w,y) \in W_k \times A_k \, : \, \varphi(y) = w\lambda_k.y \big\}, 
\]
and note that the projection $(w,y) \mapsto y$ is injective.  Indeed,  if $(w_1,y), (w_2,y) \in \widetilde{A}_k$,  then 
\[
\varphi(y) = w_1 \lambda_k.y = w_2\lambda_k.y,
\]
whence $\lambda_k^{-1}w_2^{-1}w_1\lambda_k.y = y$.  We conclude that
\[
\lambda_k^{-1}w_2^{-1}w_1\lambda_k \in \lambda_k^{-1}W_k^2 \lambda_k \cap \Lambda_a(Y) \subset U \cap \Lambda_a(Y) = \{e\}.
\]
Hence $w_1 = w_2$,  and the projection is injective.  By Lemma \ref{Lemma_Borel} (ii),  we can now find a Borel map $\rho_k : A_k \ra W_k$ such that 
$(\rho_k(y),y) \in \widetilde{A}_k$ for all $y \in A_k$.  In other words,  $\varphi(y) = \rho_k(y)\lambda_k.y$.  It remains to show that $\varphi|_{A_k}$ is injective.  We argue by contradiction and suppose $y_1,  y_2 \in A_k$ are such that $\varphi(y_1) = \varphi(y_2)$,  or equivalently,
$\rho_k(y_1)\lambda_k.y_1 = \rho_k(y_2)\lambda_k.y_2$.  Then,
\[
\lambda_k^{-1}\rho_k(y_2)^{-1}\rho_k(y_1)\lambda_k \in \Lambda_a(Y) \cap \lambda_k^{-1}W_k^2\lambda_k = \{e\},
\]
since $\lambda_k^{-1}W^2_k \lambda_k \subset V \subset U$ and $Y$ is $U$-separated.
Hence,  $\rho_k(y_1) = \rho_k(y_2)$,  which implies that the identity $\rho_k(y_1)\lambda_k.y_1 = \rho_k(y_2)\lambda_k.y_2$ just reduces to $y_1 = y_2$.  
\end{proof}

\begin{corollary}
\label{cor_invariance}
Suppose $B \subset Y$ is a $R_Y$-invariant Borel set.  Then,  
\[
\varphi^{-1}(B \cap \ran(\varphi)) = B \cap \dom(\varphi),
\]
for every (not necessarily injective) partial $R_Y$-map $\varphi$.
\end{corollary}

\begin{proof}
Let $\varphi$ be a (not necessarily injective) partial Borel $R_Y$-map.  Then,  by Lemma \ref{Lemma_partialstructure},  we can find a Borel partition $\dom(\varphi) = \bigsqcup_k A_k$ such that $\psi_k := \varphi|_{A_k}$ is an injective partial $R_Y$-map for every $k$.  Hence,
\[
\varphi^{-1}(B \cap \ran(\varphi)) = \bigsqcup_k \varphi^{-1}(B \cap \ran(\varphi)) \cap A_k
= \bigsqcup_k \psi_k^{-1}(B \cap \ran(\psi_k)).
\]
Since $B$ is $R_Y$-invariant and $\psi_k$ is injective,  $\psi_k^{-1}(B \cap \ran(\psi_k)) = B \cap \dom(\psi_k)$ for all $k$, 
whence 
\[
\varphi^{-1}(B \cap \ran(\varphi)) = \bigsqcup_k \big(B \cap \dom(\psi_k)) = B \cap \dom(\varphi).\qedhere
\]
\end{proof}

\subsection{$R_Y$-invariant sets}
Let $G$ be a lcsc group and let $(X,a, Y)$ be a $U$-transverse triple over $G$ for some open identity neighbourhood $U \subset G$. We are going to show that a subset $A \subset Y$ is invariant under the cross-section equivalence relation $R_Y$ if and only if there exists a $G$-invariant subset $B \subset X$ such that $A = B \cap Y$. One direction is in fact immediate from  Lemma \ref{Lemma_partialstructure}:

\begin{corollary}
\label{cor_RYinvariance}
If $B \subset X$ is a $G$-invariant Borel set,  then $B \cap Y \subset Y$
is a $R_Y$-invariant Borel set. 
\end{corollary}

\begin{proof}
Let $\varphi$ be a partial Borel $R_Y$-map.  By Lemma \ref{Lemma_partialstructure},
there exist a Borel partition $\dom(\varphi) = \bigsqcup_k A_k$,  Borel maps $\rho_k : A_k \ra G$ and $\lambda_k \in G$ such that
\[
\varphi(y) = \rho_k(y)\lambda_k.y,  \quad \textrm{for all $y \in A_k$}.
\]
Hence,   
\[
\varphi^{-1}(B \cap \ran(\varphi)) = \bigsqcup_k  
\varphi^{-1}(B \cap \ran(\varphi)\big) \cap A_k.
\]
Since $B$ is $G$-invariant and $\ran(\varphi) \subset Y$,  we see that
\[
\varphi(y) = \rho_k(y)\lambda_k.y \in B \cap Y
\iff 
y \in B \cap Y,
\]
for every $k$ and $y \in A_k$.  In other words,  $\varphi^{-1}(B \cap \ran(\varphi)) \cap A_k = B \cap A_k$,
for all $k$,  and thus
\[
\varphi^{-1}(B \cap \ran(\varphi))  = \bigsqcup_k B \cap A_k = B \cap \dom(\varphi).
\]
Since $\varphi$ is arbitrary,  $B \cap Y$ is $R_Y$-invariant.
\end{proof}
For the converse direction we fix a Borel $Y$-section $\beta: X \to G$. We then define a Borel map
\begin{equation}
\label{def_varphibeta}
\varphi_\beta: Y \to Y, \quad y \mapsto \beta(y).y.
\end{equation}
By construction, $\varphi_\beta$ is a (not necessarily injective) partial $R_Y$-map with $\dom(\varphi_\beta) = Y$. If $A \subset Y$ is a Borel set in $Y$, we now define $A_\beta \subset X$ by
\begin{equation}
\label{def_Abeta}
A_\beta := \big\{ x \in X \,  : \,  \beta(x).x \in A \big\}.
\end{equation}
Note that $A_\beta = b_\beta^{-1}(G \times A)$,  and thus $A_\beta$ is a Borel set in $X$. We can now establish the converse of Corollary \ref{cor_RYinvariance}:
\vspace{0.1cm}
\begin{lemma}
\label{Lemma_liftedinvariance}
Let $\beta$ be a Borel $Y$-section.  If $A \subset Y$ is a $R_Y$-invariant Borel set,  then $A_\beta \subset X$ is a $G$-invariant Borel set such that $A_\beta \cap Y = A$. 
\end{lemma}

\begin{proof}
We first note that $A_\beta \cap Y = \varphi_\beta^{-1}(A \cap \ran(\varphi_\beta))$.  
Since $A$ is $R_Y$-invariant and $\dom(\varphi_\beta) = Y$,   
Corollary \ref{cor_invariance} tells us that $\varphi_\beta^{-1}(A \cap \ran(\varphi_\beta)) = A$,  and thus $A_\beta \cap Y = A$.  To prove that $A_\beta$ is $G$-invariant,  it is enough to show (since $X = G.Y$) that 
\[
g^{-1}.A_\beta \cap h.Y = h.A,  \quad \textrm{for all $g, h \in G$}.
\]
First note that
\[
g^{-1}.A_\beta = \{ x \in X \,  : \,  \beta(g.x)g.x \in A \big\},  \quad \textrm{for all $g \in G$}
\]
and thus
\[
g^{-1}.A_\beta \cap h.Y = h.\{ y \in Y \,  : \,  \beta(gh.y)gh.y \in A \} 
= h .  \varphi_{\beta_{gh}}^{-1}(A \cap \ran \varphi_{\beta_{gh}}),   
\]
for all $g,h \in G$.  Since $A$ is $R_Y$-invariant and $\varphi_{\beta_{gh}}$ is a partial
$R_Y$-map with $\dom(\varphi_{\beta_{gh}}) = Y$,  we see that from Corollary \ref{cor_invariance} that $\varphi_{\beta_{gh}}^{-1}(A \cap \ran \varphi_{\beta_{gh}}) = A$,  and thus $g^{-1}.A_\beta \cap h.Y = h.A$ for all $g,h \in G$.
\end{proof}
In other words, the $Y$-section $\beta$ picks for every $R_Y$-invariant Borel subset $A$ a canonical $G$-invariant extension to $X$.

\subsection{Lifts of $R_Y$-invariant measures}
Let $G$ be a lcsc group and let $(X,a, Y)$ be a $U$-transverse triple over $G$ for some open identity neighbourhood $U \subset G$. The goal of this subsection is to establish the following theorem:
\vspace{0.1cm}
\begin{theorem}\label{LiftingInvariantMeasure} Let $\nu \in M(Y, \cY)^{R_Y}$ be $R_Y$-invariant and let $\beta: X \to G$ be a Borel $Y$-section.
\vspace{0.1cm}
\begin{itemize}
\item[$(i)$] The lifted measure $\nu_\beta$ is independent of the choice of $\beta$. \vspace{0.1cm}
\item[$(ii)$]  $\nu_\beta$ is $\cY_G$-finite. \vspace{0.1cm}
\item[$(iii)$] If $G$ is unimodular, then $\nu_\beta$ is $G$-invariant.
\end{itemize}
\vspace{0.1cm}
If $G$ is unimodular we thus obtain a canonical map $\ind GXY: M(Y, \cY)^{R_Y} \to M(X, \cY_G)^G$, $\nu \mapsto \nu_\beta$.
\end{theorem}

We refer to the map $\ind GXY$ from Theorem \ref{LiftingInvariantMeasure} as the \emph{induction map} of the transverse triple $(X,a,Y)$. Note that Part (ii) of the theorem is immediate from Part (i) in view of Lemma \ref{Lemma_adaptedBorelYsection} and Corollary \ref{cor_YGFinite}. The proofs of Parts (i) and (iii) are based on the following main lemma concerning piecewise equivalences of injective sets. Here, given a Borel set $C \subset G \times Y$, we set 
\[
C_y = \big\{ h \in G \,  : \,  (h,y) \in C \big\},  \quad \textrm{for $y \in Y$}.
\]
By Fubini's Theorem,  $C_y$ is a Borel set in $G$ for every $y \in Y$. 

\begin{lemma}
\label{Lemma_injectivesets}
Let $C_1, C_2 \subset G \times Y$ be injective Borel sets such that $a(C_1)= a(C_2)$.
Then there exist 
\vspace{0.1cm}
\begin{itemize}
\item a covering $\{Y^{(k)}\}$ of $Y$ by Borel sets,  Borel maps $\gamma_k : Y^{(k)} \ra G$ and $\varphi_k \in [[R_Y]]$ with $\dom(\varphi_k) = Y^{(k)}$. \vspace{0.1cm}
\item  Borel partitions $C_1 = \bigsqcup_k C_{1,k}$ and $C_2 = \bigsqcup_k C_{2,k}$ such that 
\[
C_{1,k} \subset G \times \ran(\varphi_k) \qand C_{2,k} \subset G \times \dom(\varphi_k),
\quad \textrm{for all $k$},
\]
\end{itemize} 
\vspace{0.1cm}
such that 
\[
(C_{2,k})_{y_2} = (C_{1,k})_{\varphi_k(y_2)} \gamma_k(y_2),  \quad \textrm{for all $y_2 \in Y^{(k)}$}.
\]
\end{lemma}

\begin{remark}
\label{rmk_partition}
In the proof of this lemma,  we will use a Borel partition $G = \bigsqcup V_k$ 
with the property that $V_k^{-1}V_k \cup V_k V_k^{-1} \subset U$ for all $k$.  
Such partitions can be constructed in many different ways.  Let us here briefly outline
one possible construction.  For every $g \in G$,  let $W_g$ be an open symmetric
identity neighbourhood $W_g$ such that $W_g^2$ and $gW_g^2 g^{-1}$ are 
both contained in $U$.  Then $\bigcup_{g \in G} gW_g$ is an open cover of $G$.  
Since $G$ is Lindel\"of,  there exist a countable set $\{g_k\}$ such that 
\[
G = \bigcup_k \widetilde{V}_k,  \quad \textrm{where $\widetilde{V_k} = g_k W_{g_k}$}.
\]
Note that $\widetilde{V}_k^{-1} \widetilde{V}_k \cup \widetilde{V}_k \widetilde{V}_k^{-1} \subset U$ for all $k$.  To make this collection of sets disjoint,  we set 
$V_1 := \widetilde{V}_1$ and define inductively $V_{k+1} := \widetilde{V}_{k+1} \setminus \cup_{j=1}^k  V_j$,  for $k \geq 1$ (discarding empty sets). 
\end{remark}

\begin{proof}[Proof of Lemma \ref{Lemma_injectivesets}]
We set $a_i := a|_{C_i}$ for $i = 1,2$.  By Lemma \ref{Lemma_Borel},  $a_i$ is a Borel
isomorphism between the Borel sets $C_i$ and $a(C_i)$.  Since $a(C_1) = a(C_2)$,  the composition
$\delta := a_1^{-1} \circ a_2$ is a well-defined injective Borel map and $C_1 = \delta(C_2)$.  It is not hard to
see that there is a Borel map $\theta : C_2 \ra G$ such that
\[
\delta(h_2,y_2) = (h_2 \theta(h_2,y_2)^{-1}, \theta(h_2,y_2).y_2),  \quad \textrm{for $(h_2,y_2) \in C_2$}.
\]
Let $G = \bigsqcup_k V_k$ be a Borel partition (see Remark \ref{rmk_partition}) such that 
\[
V_k^{-1} V_k \cup V_k V_k^{-1} \subset U,  \quad \textrm{for all $k$}
\]
and set
\[
C_{2,k} = \theta^{-1}(V_k^{-1}) \qand C_{1,k} = \delta(C_{2,k}).
\]
Clearly,  $C_1 = \bigsqcup_k C_{1,k}$ and $C_2 = \bigsqcup_k C_{2,k}$ are Borel partitions.  Define 
\[
Y^{(k)} = \{ y \in Y \,  : \,  Y_y \cap V_k^{-1} \neq \emptyset \} = V_k.Y \cap Y.
\]
By Lemma \ref{Lemma_BasicProp} (i),  combined with Lemma \ref{Lemma_Borel} (ii),  we see that $Y^{(k)} \subset Y$ is a Borel set,  and there is a Borel map $\gamma_k : Y^{(k)} \ra V_k^{-1}$ such that $\gamma_k(y) \in Y_y \cap V_k^{-1}$ for all $y \in Y^{(k)}$.  Pick $(h_2,y_2) \in C_{2,k}$.  Then $\theta(h_2,y_2) \in Y_{y_2} \cap V_k^{-1}$,  and thus $y_2 \in Y^{(k)}$.  However,  by Lemma \ref{Lemma_BasicProp} (i),  $|Y_{y_2} \cap V_k^{-1}| \leq 1$,  and thus 
\[
\theta(h_2,y_2) = \gamma_k(y_2),  \quad \textrm{for all $(h_2,y_2) \in C_{2,k}$}.
\]
We also define $\varphi_k(y_2) = \gamma_k(y_2).y_2$ for $y_2 \in Y^{(k)}$.  Note that 
$\varphi_k$ is injective.  Indeed,  if $\varphi_k(y_2) = \varphi_k(y_2')$,  for $y_2,  y_2' \in Y^{(k)}$,  then 
\[
\gamma_k(y_2')^{-1}\gamma_k(y_2) \in Y_{y_2} \cap V_k V_k^{-1} \subset Y_{y_2} \cap U = \{e\},
\]
since $Y$ is $U$-separated.  Hence,  $\gamma_k(y_2) = \gamma_k(y_2')$.  
Since $\varphi_k(y_2) = \varphi_k(y_2')$,  we see that $y_2 = y_2'$.  We conclude that  $\varphi_k \in [[R_Y]]$ with $\dom(\varphi_k) = Y^{(k)}$.  Furthermore,  since $C_{1,k} = \delta(C_{2,k})$,  we see that $C_{1,k} \subset G \times \ran(\varphi_k)$.  Let us now fix an index $k$ and an element $y_2 \in Y^{(k)}$ Then,  since $C_{2,k} \subset G \times Y^{(k)}$,  we see that
\begin{align*}
(C_{2,k})_{y_2} 
&=
\{ h_2 \in G \,  : \,  (h_2,y_2) \in C_{2,k} \} 
= 
\{ h_2 \in G \,  : \,  (h_2 \theta(h_2,y_2)^{-1}, \theta(h_2,y_2).y_2) \in C_{1,k} \} \\[0.1cm]
&=
\{ h_2 \in G \,  : \,  (h_2 \gamma_k(y_2)^{-1}, \varphi_k(y_2)) \in C_{1,k} \} = (C_{1,k})_{\varphi_{k}(y_2)} \,  \gamma_k(y_2).\qedhere
\end{align*}
\end{proof}
\begin{corollary}
\label{cor_measC1C2}
Let $\nu \in M_\sigma(Y)^{R_Y}$,  and suppose that $C_1,  C_2 \subset G \times Y$
are injective Borel sets such that $a(C_1) = a(C_2)$.  Then $m_G \otimes \nu(C_1) = m_G \otimes \nu(C_2)$.
\end{corollary}

\begin{proof}
Let $Y^{(k)},  C_{1,k},  C_{2,k},  \gamma_k,  \varphi_k$ be as in Lemma \ref{Lemma_injectivesets}.  Then,  for all $y_2 \in Y$,
\begin{align*}
m_G((C_2)_{y_2}) 
&= \sum_{k} m_G((C_{2,k})_{y_2}) \chi_{\dom(\varphi_k)}(y_2)
= \sum_k m_G((C_{1,k})_{\varphi_k(y_2)} \gamma_k(y_2)) \, \chi_{\dom(\varphi_k)}(y_2) \\[0.1cm]
&= 
\sum_k m_G((C_{1,k})_{\varphi_k(y_2)}) \, \chi_{\dom(\varphi_k)}(y_2),
\end{align*}
where we in the last identity have used that $m_G$ is right-invariant.  Define 
$f_k : Y \ra [0,\infty]$ by 
\[
f_k(y) = m_G((C_{1,k})_{y}),  \quad \textrm{for $y \in Y$}.
\]
By Fubini's Theorem,  $f_k$ is Borel.   Since $\varphi_k \in [[R_Y]]$ and $\nu$ is a $\sigma$-finite and $R_Y$-invariant Borel measure on $Y$,  Lemma \ref{Lemma_Integral} tells us that
\[
\int_{\dom(\varphi_k)} f_k \circ \varphi_k \,  d\nu = \int_{\ran(\varphi_k)} f_k \,  d\nu,
\]
and thus
\begin{align*}
m_G \otimes \nu(C_2) 
&=
\int_{Y} m_G((C_2)_{y_2}) \,  d\nu(y_2) = \sum_k \int_{\dom(\varphi_k)}
m_G((C_{1,k})_{\varphi_k(y_2)}) \,  d\nu(y_2) \\[0.1cm]
&=
\sum_k \int_{\ran(\varphi_k)}
m_G((C_{1,k})_{y_2}) \,  d\nu(y_2)
=
\sum_k \int_{Y}
m_G((C_{1,k})_{y_2}) \,  d\nu(y_2)
\\[0.1cm]
&= m_G \otimes \nu(C_1),
\end{align*}
since $C_1 = \bigsqcup_k C_{1,k}$ and $C_2 = \bigsqcup_k C_{2,k}$ are Borel partitions.
\end{proof}

\begin{proof}[Proof of Theorem \ref{LiftingInvariantMeasure}] 
(i) If $\beta$ and $\beta'$ are two $Y$-sections, then $b_\beta$ and $b_\beta'$ are two sections of $a|_{G \times Y}$, and hence for every Borel set $B \subset X$ the sets $C := b_\beta(B)$ and $C' := b_{\beta'}(B)$ are injective Borel sets. We now deduce with Corollary \ref{cor_measC1C2} that
\[
\nu_\beta(B) = \nu(b_\beta(B)) = \nu(C) = \nu(C') = \nu(b_{\beta'}(B)) = \nu_{\beta'}(B).
\]
(ii) Let $\cY = (Y_n)$ and let $(K_n)$ be a fundamental exhaustion of $G$. By (i) and Lemma \ref{Lemma_adaptedBorelYsection} we may assume that $b_\beta(K_n.Y_n) \subset K_n \times Y_n$ for all $n$, but then $\nu_\beta \in M(X, \cY_{\cG})$ holds by Corollary \ref{cor_YGFinite}.
(iii) If $\beta: X \to G$ is a Borel $Y$-section and $g \in G$, then we define $\beta_g(x) := \beta(g.x)g$.  Note that $\beta_g(x).x = \beta(g.x)g.x \in Y$,  whence $\beta_g$ is again a Borel $Y$-section.
For every fixed $g \in G$ we then have
\[
b_\beta(g.x) = (g \beta_g(x)^{-1},\beta_g(x).x),  \quad \textrm{for all $x \in X$}.
\]
Since $m_G$ is left-invariant,  we see that for every Borel set $B \subset X$,
\[
\nu_\beta(g.B) = m_G \otimes \nu(b_{\beta_g}(B)).  
\]
Fix a Borel set $B$ in $X$,  and let $C := b_\beta(B)$ and $C_g := b_{\beta_g}(B)$. 
Then $C$ and $C_g$ are injective Borel sets in $G \times Y$ such that $a(C) = a(C_g)$. Since $m_G$ is also right-invariant,  Corollary \ref{cor_measC1C2} tells us that  
$m_G \otimes \nu(C_g) = m_G \otimes \nu(C)$,  and thus
\[
\nu_\beta(g.B) = m_G \otimes \nu(b_{\beta_g}(B)) = m_G \otimes \nu(b_{\beta}(B)) = \nu_\beta(B).
\]
Since $g$ is arbitrary,  $\nu_\beta \in M(X;\cY_{\cG})^G$.
\end{proof}

\section{Transverse measure theory}\label{SecTransverseMeasure}
The main goal of this section is to establish Theorem \ref{GMCT}  from the introduction. Moreover we discuss ergodic decompositions of transverse measures and the behavious of transverse measure under a change of transversal. Throghout this section, $G$ denotes a lcsc group with Haar measure $m_G$ and $(X,a, Y)$ is a $U$-transverse triple over $G$ for some open identity neighbourhood $U \subset G$. We fix a fundamental exhaustion $\cG = \{ K_n\}$ of $G$ by compact sets,  and a Borel exhaustion $\cY = \{ Y_n \}$ be of $Y$ and denote by $\cY_{\cG}$ the $\cG$-suspension of $\cY$. 

\subsection{Periodizations}
Given a non-negative bounded Borel function $f$ on $G \times Y$,  we define $Tf : X \ra [0,\infty]$ by
\begin{equation}
\label{def_Tmap}
Tf(x) = \sum_{g \in Y_x} f(g^{-1},g.x),  \quad \textrm{for $x \in X$}.
\end{equation}
We refer to $Tf$ as the \emph{$Y$-periodization of $f$}.  \\
\begin{example} Let $P_o \subset G$  be a uniformly discrete subset and $\Omega_{P_o}^\times$ and $\cT_{P_o}$ as in Example \ref{CFExamples} so that by Proposition \ref{UDExamples} the triple $(X,a,Y)$ is a transverse triple. Then the $\cT_{P_o}$-periodization of a  non-negative bounded Borel function $f$ on $G \times \cT_{P_o}$ is given by
\[
Tf(P) = \sum_{p \in P} f(p^{-1}, Pp^{-1})
\]
\end{example}
\vspace{0.1cm}
In general,  periodizations have the following basic properties:
\vspace{0.1cm}
\begin{lemma}
\label{Lemma_T}
Let $f$ be a bounded non-negative Borel function on $G \times Y$. 
\vspace{0.1cm}
\begin{itemize}
\item[$(i)$] $Tf$ is Borel.  \vspace{0.1cm}
\item[$(ii)$] $T\chi_C = \chi_{a(C)}$ for every injective Borel set $C \subset G \times Y$,  \vspace{0.1cm}
\item[$(iii)$] $Tf(g.x) = Tf_g(x)$,  for all $g \in G$ and $x \in X$,  where $f_g(h,y) = f(gh,x)$. \vspace{0.1cm}
\item[$(iv)$] Suppose $f = \sum_k f_k$ pointwise,  where $f_1, f_2,\ldots$ are bounded  non-negative Borel functions on $G \times Y$.  Then $Tf = \sum_k Tf_k$ pointwise. \vspace{0.1cm}
\item[$(v)$] If $\{f > 0\} \subset K_n \times Y_n$ for some $n$,  then
\[
Tf(x) \leq M_{K_n^{-1}} \,  \|f\|_\infty \,  \chi_{K_n.Y_n}(x)\,   \quad \textrm{for all $x \in X$},
\]
where $M_{K_n^{-1}}$ is given by \eqref{DefMK}.
\end{itemize}
\end{lemma}

\vspace{0.1cm}

\begin{proof}
(i) Let $G = \bigsqcup_k V_k$ be a Borel partition such that $V_k^{-1}V_k \subset U$
for all $k$.  Then,  by Lemma \ref{Lemma_BasicProp} (i),  $a|_{V_k \times Y}$ is injective, so by Lemma \ref{Lemma_Borel} (ii),  there is a Borel right-inverse $b_k : V_k.Y \ra V_k \times Y$,  which then must be of the form $b_k(x) = (\rho_k(x)^{-1},\rho_k(x).x)$ for 
some Borel map $\rho_k : V_k.Y \ra V_k^{-1}$.  Note that $g \in Y_x \cap V_k^{-1}$ if
and only if $g = \rho_k(x)$.  Set $X_1 = V_1.Y$,  and define inductively $X_{k+1} = V_{k+1}.Y \setminus \bigcup_{j=1}^k X_j$.  Then $X_1, X_2,\ldots$ are disjoint Borel sets,  and we conclude that
\[
Tf(x) = \sum_k  \sum_{g \in Y_x \cap V_k^{-1}} f(g^{-1},g.x)\, \chi_{X_k}(x)
= \sum_k F_k(x),
\]
where $F_k(x) = f(\rho_k(x)^{-1},\rho_k(x).x)\chi_{X_k}(x)$.  Since each $F_k$ is Borel, so is $Tf$.  (ii) Suppose $C \subset G \times Y$ is an injective Borel set.  Note that 
$x \in a(C)$ if and only if there is a unique $g \in G$ such that $(g^{-1},g.x) \in C$ (in particular,  $g \in Y_x$).  Hence, 
\[
T\chi_C(x) = \sum_{g \in Y_x} \chi_C(g^{-1},g.x) = \chi_{a(C)}(x),  \quad \textrm{for all $x \in X$}.
\]
(iii) Fix $g \in G$,  and note that 
\[
Tf(g.x) = \sum_{h \in Y_{g.x}} f(h^{-1},hg.x) = \sum_{h \in Y_x} f(gh^{-1},h.x), \quad \textrm{for all $x \in X$},
\]
since $Y_{g.x} = Y_x g^{-1}$.  (iv) This readily follows from monotone convergence.  
(v) Suppose $Tf(x) > 0$.  Since $\{f > 0\} \subset K_n \times Y_n$,  there must be an
element $g \in K_n^{-1}$ such that $g.x \in Y_n$,  whence $x \in K_n.Y_n$.  In particular, 
\[
Tf(x) \leq |Y_x \cap K_n^{-1}| \|f\|_\infty \,  \chi_{K_n.Y_n}(x),  \quad \textrm{for all $x \in X$}.
\]
By Lemma \ref{Lemma_BasicProp} (ii),  there is a finite constant $M_{K_n^{-1}}$ such that $|Y_x \cap K_n^{-1}| \leq M_{K_n^{-1}}$ for all $x \in X$.
\end{proof}

\begin{remark}[On the existence of Haar measures]
Let us briefly sketch how $Y$-periodizations can be used to \emph{construct}
a Haar measure on $G$,  following the ideas of Izzo \cite{Izzo1,Izzo2}.  Suppose
$(X,a)$ is a Borel $G$-space,  $\mu$ is a $G$-invariant Borel probability measure
on $X$,  and $Y$ is a $U$-separated cross-section for some identity neighbourhood
$U$ in $G$ (if the $G$-action $a$ is free,  then one can always produce such a cross-section in $X$ by \cite[Proposition 2.10]{Forrest}; furthermore,  this construction does not use the existence of a Haar measure on $G$).  We now define 
\[
m(\varphi) = \mu(T(\varphi \otimes 1)),  \quad \textrm{for $\varphi \in C_c(G)$}.
\]
The proof of Lemma \ref{Lemma_T} above now shows that $m$ is a left $G$-invariant locally finite Borel measure on $G$.  In particular,  the existence of a single free probability-measure preserving $G$-action implies the existence of a (left-invariant) Haar measure on $G$.
\end{remark}

\subsection{Existence of transverse measures}

\begin{proposition}
\label{Prop_Tmap}
For every $\mu \in M(X;\cY_{\cG})^G$,  there is a unique $\nu \in M(Y;\cY)$ such that
\begin{equation}\label{muTf}
\mu(Tf) = (m_G \otimes \nu)(f),  
\end{equation}
for every bounded non-negative Borel function $f$ on $G \times Y$.  In particular,  $\mu(a(C)) = m_G \otimes \nu(C)$ for every injective Borel set $C \subset G \times Y$.
\end{proposition}

\begin{remark}
\label{Remark_muYfinite}
We refer to $\nu$ as the \emph{$Y$-transverse measure} of $\mu$.  If we want to emphasize its dependence on $\mu$,  we use the notation $\nu = \mu_Y$.  We stress that we are \emph{not} assuming in this lemma that $G$ is unimodular. We can spell out \eqref{muTf} in the case where $f$ is a function which depends only on one of the two variables: If $f_1$ and $f_2$ are bounded non-negative Borel functions on $G$ and $Y$ respectively, then
\[
 \int_G f_1(g) \, dm_G(g) = \int_X \sum_{g \in Y_x} f_1(g^{-1}) \, d\mu(x) \qand \int_Y f_2(y)\, d\nu(y) =  \int_X \sum_{g \in Y_x} f_2(g.x)\, d\mu(x).
\]
Note that $\mu$ is a finite Borel measure,  then $\nu$ is a finite Borel measure on $Y$.  The converse is not true.  
\end{remark}
\begin{proof}[Proof of Proposition \ref{Prop_Tmap}]
Fix $\mu \in M(X;\cY_{\cG})^G$,  and define
\[
\eta(D) := \mu(T\chi_D),  \quad \textrm{for Borel sets $D \subset G \times Y$}.
\]
By Lemma \ref{Lemma_T} (i) and (iv),  $\eta$ is a Borel measure on $G \times Y$.  
Furthermore,  Lemma \ref{Lemma_T} (v) implies that $\eta(K_n \times Y_n) < \infty$ for all $n$,  so in particular,  $\eta$ is $\sigma$-finite.  By a standard approximation argument (see e.g.  \cite[Theorem 3.3.1]{Bog}),  it suffices to show that there is a unique 
$\nu \in M(Y;\cY)$ such that
\[
\eta(A \times B) = m_G(A) \,  \nu(B),  \quad \textrm{for all $A \in \mathscr{B}_G$ and $B \in \mathscr{B}_Y$}.
\]
To do this,  fix $n$ and a Borel set $B \subset Y_n$.  We define the Borel measure $\eta_B$ on $G$ by
\[
\eta_B(A) = \eta(A \times B),  \quad \textrm{for $A \in \mathscr{B}_G$}.
\]
Since $\mu$ is $\cY_{\cG}$-finite,  $\eta_B(K_n) \leq M_{K_n^{-1}} \mu(K_n.Y_n) < \infty$ for every $n$ by Lemma \ref{Lemma_T} (v),  and thus $\eta_B$ is locally finite.  By Lemma \ref{Lemma_T} (iii),  we see that
\[
\eta_B(g^{-1}A) = \mu(T\chi_{A \times B}(g.\cdot)) = \eta_B(A),  \quad \textrm{for all $g \in G$},
\]
since $\mu$ is $G$-invariant.  Hence $\eta_B$ is a locally finite and left-invariant Radon measure on $G$,  and thus a (non-negative) multiple of $m_G$.  We conclude that there is a non-negative (finite) number $\nu_n(B)$ such that 
\[
\eta(A \times B) = m_G(A) \,  \nu_n(B),  \quad \textrm{for all $A \in \mathscr{B}_G$}.
\]
Since $\eta$ is a Borel measure,  it follows that $\nu_n$ is a (finite) Borel measure on $Y_n$.  Furthermore,  if $m \leq n$,  then $\nu_m(B) = \nu_n(B)$ for every Borel set $B \subset Y_m$.  We conclude that the limit
\[
\nu(B) := \lim_{n \ra \infty} \nu_n(B \cap Y_n), 
\]
exists for every Borel set $B \subset Y$.  It readily follows from monotone convergence that $\nu$ is a Borel measure on $Y$ such that $\nu(B) = \nu_n(B)$
for every Borel set $B \subset Y_n$.  In particular,  $\nu$ is $\cY$-finite.  Since 
$\eta$ is a Borel measure,  we now have
\[
\eta(A \times B) = \lim_n \eta(A \times (B \cap Y_n)) = \lim_n m_G(A) \nu_n(B \cap Y_n) = m_G(A) \nu(B),
\]
for all Borel sets $A \subset G$ and $B \subset Y$.  
\end{proof}

\subsection{Proofs of Theorem \ref{GMCT} and Corollary \ref{CorRadon}}

We now turn to the proof of Theorem \ref{GMCT}. We split the proof into a series of lemmas. For the first two of these, unimodularity of $G$ is not needed.

\begin{lemma}
\label{Lemma_injectivemap}
The map $M(X;\cY_G)^G \ra M(Y;\cY)$, $\mu \mapsto \mu_Y$ is injective.  In particular,  if $\mu$ is non-zero,  then so is $\mu_Y$.
\end{lemma}

\begin{proof}
Suppose $\mu_1,  \mu_2 \in M_{\cY}(X)^G$ with $(\mu_1)_Y = (\mu_2)_Y$.  Then, 
\[
\mu_1(a(C)) = m_G \otimes (\mu_1)_Y(C) = m_G \otimes (\mu_2)_Y(C) = \mu_2(a(C))
\]
for every injective Borel set $C \subset G \times Y$.  Fix a Borel set $B \subset X$.  By Lemma \ref{Lemma_BasicProp} (v),  there are injective
Borel sets $\widetilde{B}_1,\widetilde{B}_2,\ldots$ in $G \times Y$ such that $B = \bigsqcup_k a(\widetilde{B}_k)$.   Since
$\mu_1(a(\widetilde{B}_k)) = \mu_2(a(\widetilde{B}_k))$
for all $k$,  we see that
\[
\mu_1(B) = \sum_k \mu_1(a(\widetilde{B}_k)) = \sum_k \mu_2(a(\widetilde{B}_k))
= \mu_2(B). 
\]
Since $B$ is an arbitrary Borel set in $X$,  we conclude that $\mu_1 = \mu_2$. 
\end{proof}

\begin{lemma}
\label{lemma_null}
Let $\mu \in M(X;\cX)^{G}$ and let $B$ be a $G$-invariant Borel set in $X$.  Then, 
\[
\mu(B) = 0 \iff \mu_Y(B \cap Y) = 0.
\]
\end{lemma}

\begin{proof}
Let $V \subset G$ be an identity neighbourhood such that $V^{-1}V \subset U$. 
Then $V \times Y$ is an injective Borel set by Lemma \ref{Lemma_BasicProp} (i).  
Since $B$ is $G$-invariant,  $V.(B \cap Y) = B \cap V.Y$,  and thus
\[
m_G(V) \mu_Y(B \cap Y) = \mu(V.(B \cap Y)) = \mu(B \cap V.Y).
\]
Since $0 < m_G(V) < \infty$,  we see that $\mu_Y(B \cap Y) = 0 \iff \mu(B \cap V.Y) = 0$.  In particular,  if $\mu(B) = 0$,  then $\mu_Y(B \cap Y) = 0$.  Conversely,  if $\mu_Y(B \cap Y) = 0$,  then
$\mu(B \cap V.Y) = 0$,  so if we let $\Xi \subset G$ be a countable set such that $\Xi V = G$,  then,  since $\mu$ and $B$ are $G$-invariant,  $\mu(B \cap \xi V.Y) = 0$ for all $\xi \in \Xi$.  Hence,  by $\sigma$-additivity of $\mu$,  we see that $0 = \mu(B \cap \Xi V.Y) = \mu(B \cap G.Y) = \mu(B)$.
\end{proof}

\begin{lemma}
\label{Lemma_RYinvariance}
If $G$ is unimodular and $\mu \in M(X;\cY_{\cG})^G$,  then $\mu_Y \in M(Y;\cY)^{R_Y}$ is $R_Y$-invariant.
\end{lemma}

\begin{proof}
Let $\varphi \in [[E]]$,  and let $V,A_k$,  $W_k,\rho_k,  \lambda_k$ and $C_k$ be as in Lemma \ref{Lemma_partialstructure}.  Since $\varphi$ is injective and the sets $A_k$ are disjoint,  we have
\[
\mu_Y(\varphi(A)) = \sum_{k} \mu_Y(\varphi(A_k)),
\]
so it suffices to show that $\mu_Y(\varphi(A_k)) = \mu_Y(A_k)$ for every $k$.  We recall from Lemma \ref{Lemma_partialstructure} that
\[
C_k := \big\{ (w\rho_k(y)^{-1},\varphi(y)) \,  : \,  w \in W_k,  \enskip y \in A_k \big\} \subset G \times Y
\]
is an injective set such that $a(C_k) = W_k\lambda_k.A_k$.  Let $\widetilde{\rho}_k := \rho \circ \varphi^{-1}|_{\varphi(A_k)}$,  and note that
\[
\chi_{C_k}(g,z) = \chi_{W_k}(g\widetilde{\rho}_k(z)) \chi_{\varphi(A_k)}(z),  \quad \textrm{for all $(g,z) \in G \times Y$}.
\]
Since $m_G$ is \emph{right}-invariant,  Fubini's Theorem tells us that
$m_G \otimes \mu_Y(C_k) = m_G(W_k) \mu_Y(\varphi(A_k))$.  
Since $C_k$ is injective,  we also have $(m_G \otimes \mu_Y)(C_k) = \mu(a(C_k))$.  
Moreover,  
\begin{align*}
\mu(a(C_k)) &= \mu(W_k\lambda_k.A_k) = \mu(\lambda_k^{-1}W_k\lambda_k.A_k) \\[0.1cm]
&= m_G(\lambda_k^{-1}W_k \lambda_k) \mu_Y(A_k) = m_G(W_k) \mu(A_k),
\end{align*}
where in the second identity we have used that $\mu$ is $G$-invariant,  in the third inequality we have used that $\lambda_k^{-1} W_k^{2} \lambda_k \subset U$,  and in the last identity we have used that $m_G$ is conjugation-invariant.  Hence
\[
m_G(W_k) \mu_Y(\varphi(A_k)) = m_G \otimes \mu_Y(C_k) = \mu(a(C_k))  = m_G(W_k) \mu(A_k),
\]
and thus $\mu_Y(\varphi(A_k)) = \mu_Y(A_k)$. 
\end{proof}
Thus, if we assume that $G$ is unimodular, then we have an injective restriction map
\[
\res G X Y: M(X;\cY_{\cG})^G \ra M(Y;\cY)^{R_Y}, \quad \mu \mapsto \mu_Y.
\]
It turns out that this map is actually a bijection:

\begin{lemma}\label{Lemma_bijection} If $G$ is unimodular, then the map $\res G X Y: M(X;\cY_{\cG})^G \ra M(Y;\cY)^{R_Y}$, $\mu \mapsto \mu_Y$ is bijective and inverse to the map $\ind G X Y: M(Y;\cY)^{R_Y} \ra  M(X;\cY_{\cG})^G$ from Theorem \ref{LiftingInvariantMeasure}.
\end{lemma}
\begin{proof} We fix a Borel $Y$-section $\beta$ as in Lemma \ref{Lemma_adaptedBorelYsection}; then $\ind G X Y(\nu) = \nu_\beta  \in M(X;\cY_{\cG})^G$ for all $\nu \in M(Y;\cY)^{R_Y}$ by Theorem \ref{LiftingInvariantMeasure}. Since $\res GXY$ is injective, it now suffices to show that 
\begin{equation}\label{LemmabijectionToShow}
(\nu_\beta)_Y = \nu \quad \text{for all }\nu \in M(Y;\cY)^{R_Y};
\end{equation}
this will prove that $\res GXY$ is surjective (hence bijective) and hence that its left-inverse $\ind G X Y$ is actually its inverse.

To prove \eqref{LemmabijectionToShow} we fix a Borel set $B \subset Y$.  Let $V$ be an identity neighbourhood in $G$ such that $V^{-1}V \subset U$.  Then
\[
C_1 = b_\beta(V.B) \qand C_2 = V \times B,
\]
where $b_\beta$ is defined in \eqref{def_bbeta},  are both injective Borel sets in $G \times Y$ and we have
\[
a(C_1) = V.B = a(C_2)
\]
Since $\nu$ is $R_Y$-invariant,  it follows from Corollary \ref{cor_measC1C2} that $m_G \otimes \nu(C_1) = m_G \otimes \nu(C_2)$,  and thus 
\[
\nu_\beta(V.B) = m_G \otimes \nu(C_1)  = m_G(V) \nu(B). 
\]
In particular,  $(\nu_\beta)_Y(B) = \nu(B)$.  Since $B$ is an arbitrary Borel set in $Y$,  
we see that $(\nu_\beta)_Y = \nu$.
\end{proof}

At this points we have established the measure correspondence promised in the introduction:
\begin{proof}[Proof of Theorem \ref{GMCT}] The maps $\res G X Y$ from Lemma \ref{Lemma_bijection} and $\ind G X Y$ from Theorem \ref{LiftingInvariantMeasure} are mutually inverse by Lemma \ref{Lemma_bijection}. It remains to show that if $\mu$ and $\nu$ are as in the theorem, then (i)-(iv) hold. Property (ii) is actually our definition of $\nu$ and (i) follows from (ii) (see Proposition \ref{Prop_Tmap} and Remark \ref{Remark_muYfinite}). Moreover, (iv) was established in Lemma \ref{lemma_null}. If $C$ is injective then by Lemma~\ref{Lemma_T}~(ii) we have $T\chi_C = \chi_{a(C)}$ and hence
\[
\mu(a(C)) = \mu(\chi_{a(C)}) = \mu(T\chi_C) = (m_G \otimes \nu)(\chi_C) = (m_G \otimes \nu)(C).
\]
This establishes (iii) and finishes the proof.
\end{proof}
\begin{proof}[Proof of Corollary \ref{CorRadon}] Since $Y$ is closed in $X$, it is locally compact, hence we can find an exhaustion $\mathcal Y = \{Y_n\}$ of $Y$ by compact sets with non-empty interior. Then the Radon measures on $Y$ are precisely the measures in $M(Y, \cY)$. If $\cG = \{K_n\}$, then the sets $X_n := K_nY_n$ are compact and exhaust $X$. Since $X$ is Baire, it follows that almost all $X_n$ contain an interior point, but this in turn implies that the measures in $M(X, \cY_\cG)$ are precisely the Radon measure on $X$. The corollary then follows from Theorem \ref{GMCT}.
\end{proof}

\subsection{Ergodic decomposition of $Y$-transverse measures}
In this subsection we assume that $G$ is unimodular so that by Lemma \ref{Lemma_RYinvariance} the transverse measure $\mu_Y$ is $R_Y$-invariant for every $G$-invariant measure $\mu \in M(X;\cY_{\cG})^G$. We first observe that the restriction map preserves ergodicity in the following sense:
\vspace{0.1cm}

\begin{lemma}
\label{Lemma_Ergodic}
Suppose $\mu \in M(X;\cY_{\cG})^G$.  Then
\[
\textrm{$\mu$ is $G$-ergodic} \iff \textrm{$\mu_Y$ is $R_Y$-ergodic}.
\]
\end{lemma}

\begin{proof}
Suppose $\mu$ is $G$-ergodic,  and let $A \subset Y$ be a $R_Y$-invariant subset
such that $\mu_Y(A) > 0$.  We want to show that $\mu_Y(A^c) = 0$.  Let us fix a Borel 
$Y$-section $\beta : X \ra G$ and define 
\[A_\beta = \{ x \in X \,  : \,  \beta(x).x \in A \}.
\] 
By Lemma \ref{Lemma_liftedinvariance},  $A_\beta \subset X$ is a $G$-invariant Borel set such that $A_\beta \cap Y = A$.  Since $\mu_Y(A) > 0$,  it follows from Lemma \ref{lemma_null} that $\mu(A_\beta) > 0$.   Since $\mu$
is ergodic,  we have $\mu(A_\beta^c) = 0$.  Using Lemma \ref{lemma_null} again,  and
the identity $(A_\beta)^c = (A^c)_\beta$,  we see that
\[
\mu_Y(A_\beta^c \cap Y) = \mu_Y(A^c) = 0.
\]
To prove the converse,  assume that $\mu_Y$ is $R_Y$-ergodic,  and let $B \subset X$ be a $G$-invariant Borel set such that $\mu(B) > 0$.  We want to show that $\mu(B^c) = 0$.  By Corollary \ref{cor_RYinvariance},  $B \cap Y$ is a $R_Y$-invariant Borel set in $Y$,  and by \ref{lemma_null},  we have $\mu_Y(B \cap Y) > 0$.  Since 
$\mu_Y$ is $R_Y$-ergodic,  $\mu_Y(B^c \cap Y) = 0$,  and thus 
$\mu(B^c) = 0$ by Lemma \ref{lemma_null}.  
\end{proof}
This allows us to carry over the ergodic decomposition of invariant measures in $M(X;\cY_{\cG})$ to transverse measures. We will use the following version of the ergodic decomposition theorem which is  contained in \cite[Theorem 1.1]{GS}.

\begin{theorem}
\label{Thm_ErgodicDecomposition}
For every $\mu \in \Prob(X)^G$,  there is a unique probability measure $\sigma$ on
$\Prob(X)$,  which is supported on the set of $G$-ergodic Borel probability measures on $X$,  such that
\[
\mu(F) = \int_{\Prob(X)^G} \eta(F) \, d\sigma(\eta), 
\] 
for every bounded real-valued Borel function $F$ on $X$.\qed
\end{theorem}

If we now consider $F$ of the form $T(\rho \otimes f)$,  where $T$ is defined as in \eqref{def_Tmap},  $\rho$ is a compactly supported continuous function with $m_G(\rho) = 1$ and $f$ is a bounded Borel function on $Y$,  then Proposition \ref{Prop_Tmap} and Lemma \ref{Lemma_Ergodic} give the following corollary. 

\begin{corollary}
\label{Cor_ErgodicDecomposition}
For every $\mu \in \Prob(X)^G$,  there is a unique probability measure $\sigma$
on $\Prob(X)$,  which is supported on the set of $G$-ergodic Borel probability measures on $Y$,  such that
\[
\mu_Y(f) = \int_{\Prob(X)^{G}} \eta_Y(f) \,  d\sigma(\eta),
\]
for every bounded real-valued Borel function $f$ on $Y$.\qed
\end{corollary}

\subsection{Change of cross-section}
We continue to assume that $G$ is unimodular. We observe that since $Y \subset X$ is a $U$-separated cross section, the translate $g.Y$ is a $U^g$-separated cross-section. The corresponding transverse measures with respect to $Y$ and $g.Y$ are related as follows:
\begin{lemma}
\label{Lemma_YgY}
Let $\mu \in M_{\textrm{fin}}(X)^G$.  Then,  for every $g \in G$,  we have $g_*\mu_Y = \mu_{g.Y}$.
\end{lemma}

\begin{remark}
The lemma,  with virtually the same proof,  also holds for $\cY_{\cG}$-finite Borel measures,  for a Borel exhaustion $\cY$ of $Y$ and a fundamental exhaustion $\cG$ of $G$ by compact sets. 
\end{remark}

\begin{proof}
Fix $g \in G$ and let $V$ be an identity neighbourhood in $G$ such that $V^{-1}V \subset U$.  Suppose $B \subset g.Y$ is a Borel set.  Then,  $g^{-1}.B$ is a Borel set in $Y$,  and 
\[
g_*\mu_Y(B) = \mu_Y(g^{-1}.B) = \frac{\mu(Vg^{-1}.B)}{m_G(V)}.
\]
Since $\mu$ is $G$-invariant and $V^g (V^g)^{-1} \subset U^g$ and $g.Y$ is $U^g$-separated,   we see that
\[
\frac{\mu(Vg^{-1}.B)}{m_G(V)} = \frac{\mu(gVg^{-1}.B)}{m_G(V)} = \frac{m_G(gVg^{-1})}{m_G(V)} \,  \mu_{g.Y}(B).
\]
Since $G$ is unimodular and $B \subset g.Y$ is arbitrary,  we conclude that $g_*\mu_Y = \mu_{g.Y}$.
\end{proof}

\section{Transversal recurrence}\label{SecTransversalRecurrence}

In this section we establish a transverse recurrence theorem which corresponds to the case $r=1$ of Theorem \ref{MTR3} from the introduction. The versions for $r>1$ will later be established by applying this theorem to a suitable intersection space. 

\subsection{A transverse version of Poincar\'e's Recurrence Theorem}

Let $G$ be a  \emph{non-compact} lcsc group with left Haar measure $m_G$.  Let
$X$ be a Polish space,  and denote by $\mathscr{B}_X$ the $\sigma$-algebra on $X$ generated by the open
sets.  Let $a : G \times X \ra X$ be a Borel measurable action,  and suppose $Y \subset X$ is a $U$-separated cross-section for some identity neighbourhood $U$ in $G$.  The aim of this section is to prove the following transversal version of Poincar\'e's classical recurrence theorem.  
\begin{theorem}\label{Thm_Poincare}
For every $\mu \in M_{\textrm{fin}}(X)^G$,  there is a $\mu_Y$-conull Borel set $Y' \subset Y$ such that for every $y \in Y'$,  there is a sequence $(g_n) \in Y_y$ such
that $g_n.y \ra y$ and $g_n \ra \infty$ as $n \ra \infty$.
\end{theorem}
In other words, for a generic point in the transversal (with respect to the transverse measure) we can find an infinite number of return times to an arbitrary small neighbourhood in $Y$ (rather than in $X$). The assumption that $X$ is Polish is used in the proof to construct an exhaustion of $X$ by compact sets and to guarantee that $X$ is second countable.

\subsection{Proof of Theorem \ref{Thm_Poincare}}

We will prove Theorem \ref{Thm_Poincare} by reducing it to the non-transversal version of Poincar\'e's recurrence theorem. Throughout we fix $\mu \in M_{\textrm{fin}}(X)^G$. A subset $\Xi \subset G \setminus \{e\}$ is \emph{Poincar\'e} (with respect to the triple $(X,a,\mu)$) if for every Borel set $A \subset X$ with positive $\mu$-measure,  there exists $\xi \in \Xi$ such that $\mu(A \cap \xi^{-1}A) > 0$.  We can then state Poincar\'e's recurrence theorem as follows:
\begin{lemma}[Poincar\'e's Recurrence Theorem]
Let $\Theta \subset G$ be an infinite set.  Then $\Xi := \Theta \,  \Theta^{-1} \setminus \{e\}$ is a Poincar\'e set.  
\end{lemma}

\begin{proof}
Assume,  for the 
sake of contradiction,  that there exists a Borel set $A \subset X$ with positive $\mu$-measure such that 
\[
\mu(A \cap \theta_i \theta_j^{-1} .A) = \mu(\theta_i^{-1}.A \cap \theta_j^{-1}.A) = 0, \quad \textrm{for all $i \neq j$},
\]
where $\theta_1, \theta_2,\ldots$ is an infinite sequence of distinct elements in 
$\Theta$.  Then the sets $\theta_1^{-1}.A,  \theta_2^{-1}.A,\ldots$ are disjoint modulo 
$\mu$-null sets,  and thus
\[
\infty = \sum_i \mu(A) = \sum_i \mu(\theta_i^{-1}.A) = \mu\Big(\bigcup_i \theta_i^{-1}.A\Big) \leq 1.\qedhere
\]
\end{proof}

\begin{corollary}
\label{Cor_Poincare}
For every compact set $K$ in $G$,  there is a countable Poincar\'e set $\Xi$ such that $\Xi \cap K = \emptyset$.
\end{corollary}

\begin{proof}
Fix a compact set $K$ in $G$.  Since $G$ is non-compact,  there is an infinite sequence $\Theta = (\theta_i)$ in $G$ such that $\theta_i^{-1}K \cap \theta_j^{-1}K = \emptyset$ for all distinct $\theta_i, \theta_j \in \Theta$.  By
the previous lemma,  $\Xi = \Theta \Theta^{-1} \setminus \{e\}$ is a (countable) Poincar\'e set,  and clearly $\Xi \cap K = \emptyset$.
\end{proof}

The proof of Theorem \ref{Thm_Poincare} will be based on two lemmas. If $\Xi$ is a countable Poincar\'e set in $G$, and $A \subset X$ is a Borel set,  we denote by
\[
A_{\Xi} := \big\{ x \in A \,  : \,  \xi.x \in A,  \enskip \textrm{for some $\xi \in \Xi$}\big\}
\]
the set of all points in $A$ whose return time sets to $A$ contains $\Xi$. Since $A_{\Xi} = \bigcup_{\xi \in \Xi} A \cap \xi^{-1}.A$,  and $\Xi$ is countable,  we
see that $A_{\Xi}$ is a Borel set in $X$. 

\begin{lemma}
\label{Lemma_AXi}
If $A \subset X$ is a Borel set with $\mu(A)>0$,  then $A_{\Xi}$ is a $\mu$-conull Borel subset of $A$.
\end{lemma}

\begin{proof}
Set $B := A \setminus A_{\Xi}$,  and assume, for the sake of contradiction, that $\mu(B) > 0$.  Since $\Xi$ is a Poincar\'e set,  this implies that there exists $\xi \in \Xi$
such that $\mu(B \cap \xi^{-1}B) > 0$.  In particular,  $B \cap \xi^{-1}.B$ is non-empty, which contradicts the fact that $B \subset A$ and $B \cap A_{\Xi} = \emptyset$.   
\end{proof}

For the second lemma we consider a Borel set $V\subset G$ such that $V^{-1}V \subset U$. By Lemma \ref{Lemma_BasicProp} (i),  the composition
\begin{equation}
\label{def_prY}
\pr_Y : V.Y \ra V \times Y \ra Y, \enskip v.y \mapsto (v,y) \mapsto y
\end{equation}
is well defined and Borel.  
\begin{lemma}
\label{Lemma_prYX}
Let $X' \subset X$ be a $\mu$-conull Borel set.  Then $\pr_Y(X' \cap V.Y)$ contains a $\mu_Y$-conull Borel set.
\end{lemma}

\begin{proof}
Since both $a$ and $\pr_Y$ are Borel,  \cite[Proposition 14.4]{Kechris} tells us that 
$Y' := \pr_Y(X' \cap V.Y)$ and $V.Y'$ are analytic sets in $Y$ and $X$ respectively. Hence,  by \cite[Theorem 21.10]{Kechris},  $Y'$ and $V.Y'$ are measurable with respect to the $\mu_Y$-completion of $\mathscr{B}_Y$ and the $\mu$-completion of $\mathscr{B}_X$ respectively (we abuse notation and denote by $\mu_Y$ and $\mu$ the unique extensions of $\mu_Y$ and $\mu$ to the respective completions).  Note that $V.Y' \supset X' \cap V.Y$.  Since $X'$ is $\mu$-conull we have
\[
\mu(X' \cap V.Y) = \mu(V.Y) = m_G(V) \mu_Y(Y),
\]
and thus 
\[
\mu_Y(Y') = \frac{\mu(V.Y')}{m_G(V)} \geq 
\frac{\mu(X' \cap V.Y)}{m_G(V)} = \mu_Y(Y).
\]
Since $Y'$ is measurable with respect to the $\mu_Y$-completion of $\mathscr{B}_Y$,  we can find a Borel set $Y'' \subset Y'$ which is still $\mu_Y$-conull.
\end{proof}
The proof still applies if $\mu$ is merely assumed to be $\sigma$-finite, but we will not need this fact here.

\begin{proof}[Proof of Theorem \ref{Thm_Poincare}]

We fix an identity neighbourhood $V$ in $G$ such that $V^{-1}V \subset U$,  and an exhaustion 
$\{ K_n \,  : \,  n \geq 1 \}$ of $G$ by compact sets.  By Corollary \ref{Cor_Poincare},
there exists a sequence $(\Xi_n)$ of countable Poincar\'e sets in $G$ such that $\Xi_n \cap K_n = \emptyset$ for all $n$.  Since $X$ is Polish,  it is second countable.  Fix a countable basis for the topology on $X$ and denote by $\mathfrak{B} := \{B\}$ the 
collection of (non-empty) intersections of the basis elements with $\supp(\mu_Y) \subset Y$.  Define
\[
\widetilde{B} := V.B \subset X,  \quad \textrm{for $B \in \mathfrak{B}$}.
\]
Since $\mu_Y(B) > 0$,  we have $\mu(\widetilde{B}) > 0$,  for every $B \in \mathfrak{B}$.  Define
\[
N_{B,n} := \widetilde{B} \setminus \widetilde{B}_{\Xi_n} \qand X' := \bigcup_{B \in \mathfrak{B}} \bigcup_{n = 1}^\infty X \setminus N_{B,n}.
\]
By Lemma \ref{Lemma_AXi} we have  $\mu(N_{B,n}) = 0$ for all $B \in \mathfrak{B}$ and $n \in \mathbb N$.  Since $\mathfrak{B}$ is countable,  $X'$ is thus a $\mu$-conull Borel set.  By Lemma \ref{Lemma_prYX} the set $\pr_Y(X' \cap V.Y)$ contains a $\mu_Y$-conull Borel subset 
\[
Y' \subset \pr_Y(X' \cap V.Y) \subset Y,
\]
We claim that $Y'$ satisfies the conclusions of Theorem \ref{Thm_Poincare}.  To see this,  pick 
\[
y \in Y',  \quad v \in V \qand x = v.y \in X' \cap V.Y,
\]
and a sequence $(B_n)$ in $\mathfrak{B}$ such that $\displaystyle \{y\} = \bigcap_{n} B_{n}$.  

Since $x \in X' \cap V.B_n$ for every $n$,  and thus $x \in (V.B_n)_{\Xi_n}$,  we can find $\xi_n \in \Xi_n$
such that $\xi_n.x \in V.B_n$.  Hence there exist $v_n \in V$ and $y_n \in B_n$ such that
\[
\xi_nv.y = v_n.y_n,
\]
and thus $g_n := v_n^{-1}\xi_n v \in Y_y \cap V \xi_n V$ and $g_n.y = y_n \in B_n$. 
Since $\xi_n \notin K_n$ for all $n$,  we see that $g_n \ra \infty$, and since $\displaystyle \{y\} = \bigcap_n B_n$,  we see that $y_n \ra y$.
\end{proof}

\section{Ergodic theorems for cross-sections}\label{SecErgodicTransversal}

In this section we establish a pointwise transversal ergodic theorem along certain averaging sequences which we call \emph{convenient sequences}. As an application we deduce that generic points in the transversal have positive lower density along such sequences. This establishes Theorem \ref{PD3} from the introduction in the case $r=1$. The versions for $r>1$ will later be obtained by applying this version to a suitable intersection space. Throghout this section, $G$ denotes a unimodular lcsc group with bi-invariant Haar measure $m_G$ and $(X,a, Y)$ is a $U$-transverse triple over $G$ for some open identity neighbourhood $U \subset G$. 

\subsection{Convenient sequences}
If $V$ is an open identity neighbourhood and $B \subset G$ is a pre-compact Borel set,  then we set
\[
B^{-}_{V} := \bigcap_{v \in V} v^{-1}B \qand  B^{+}_{V} := \bigcup_{v \in V} v^{-1}B.
\]
\begin{definition}\label{DefConvenient}
Let $(V_n)$ be a decreasing sequence of open identity neighbourhoods in $G$.  We say that a sequence $(G_t)$ of pre-compact Borel sets in $G$ is \emph{convenient} if
\begin{itemize}
\item[$(i)$] there exist sequences $(\delta_n),  (\eps_n)$ and $(t_n)$ of positive real numbers such that $\delta_n,  \eps_n \ra 0$,  and
\begin{equation}
\label{contain}
G_{t-\delta_n} \subset (G_t)_{V_n}^{-} \subset (G_t)_{V_n}^{+} \subset G_{t+\delta_n},  \quad \textrm{for all $t > t_n$}
\end{equation}
and
\begin{equation}
\label{ratios}
1-\eps_n \leq \varliminf_{t \ra \infty} \frac{m_G(G_{t-\delta_n})}{m_G(G_t)} 
\qand
\varlimsup_{t \ra \infty} \frac{m_G(G_{t+\delta_n})}{m_G(G_t)} \leq 1 + \eps_n,
\end{equation}
for all $n$.  \vspace{0.1cm}

\item[$(ii)$] for every Borel $G$-space $(X,a)$ and bounded real-valued Borel function $\varphi$ on $X$,  the set
\[
E_\varphi := \Big\{ x \in X \,  : \,  \lim_{t \ra \infty} \frac{1}{m_G(G_t)} \int_{G_t} \varphi(g.x) \,  dm_G(g) \enskip \textrm{exists} \Big\},
\]
is Borel in $X$ and $\mu$-conull for every $\mu \in M_{\textrm{fin}}(X)^G$.  Furthermore,  the function $\overline{\varphi} : E_\varphi \ra \bR$ defined by
\[
\overline{\varphi}(x) = \lim_{t \ra \infty} \frac{1}{m_G(G_t)} \int_{G_t} \varphi(g.x) \,  dm_G(g),  \quad \textrm{for $x \in E_\varphi$}
\]
is Borel,  and for every $G$-ergodic $\mu \in \Prob(X)^G$,  there is a $G$-invariant and
$\mu$-conull Borel subset $E_\varphi(\mu) \subset E_\varphi$ such that $\overline{\varphi} = \mu(\varphi)$ for all $x \in E_\varphi(\mu)$.
\end{itemize}
\end{definition}

\begin{remark}
The key point of Condition (ii) is that the set $E_\varphi(\mu)$ is $G$-invariant.  While this is automatic if $G$ is amenable,  and $(G_t)$ is a sufficiently nice F\o lner sequence in $G$,  it is not at all automatic when $G$ is non-amenable.  Sufficient conditions on $(G_t)$ to satisfy (ii) are given in \cite[Theorem 5.22]{GorodnikNevo}.  (note however that the averages in this book are taken over $G_t^{-1}$).  Examples of convenient sequences in semisimple Lie groups are
explicated in \cite[Chapter 7]{GorodnikNevo}.
\end{remark}

\subsection{A pointwise transversal ergodic theorem}
Our main goal in this section is to prove the following theorem.
\begin{theorem}
\label{Thm_Ergodic}
Suppose $\mu \in \Prob(X)^G$ is $G$-ergodic.  Then,  for every convenient sequence $(G_t)$ pre-compact Borel sets in $G$ and for every bounded real-valued Borel function $f$ on $X$,  there is a $\mu_Y$-conull Borel set $E_f \subset Y$ such
that 
\[
 \lim_{t \ra \infty} \frac{1}{m_G(G_t)} \sum_{h \in Y_y \cap G_t} f(h.y) = \mu_Y(f), \quad \textrm{for all $y \in E_f$}.
\]
\end{theorem}
\begin{proof}
We may without loss of generality assume that $f$ is non-negative.  Given a non-negative continuous function $\rho$ with compact support such that $m_G(\rho) = 1$,
we define
\[
f_\rho(x) = \sum_{h \in Y_x} \rho(h^{-1}) f(h.x),  \quad x \in X.
\]
Note that $f_\rho = T(\rho \otimes f)$,  where $T$ is the $Y$-periodization defined in \eqref{def_Tmap},  so in particular $f_\rho$ is a bounded Borel function on $X$ by Lemma \ref{Lemma_T} (i) and (v),  and by Proposition \ref{Prop_Tmap},  
\begin{equation}
\label{muf}
\mu(f_\rho) = m_G(\rho) \mu_Y(f) = \mu_Y(f),  \quad \textrm{for every $\mu \in M_{\textrm{fin}}(X)^G$}.
\end{equation} 
Furthemore,  by Lemma \ref{Lemma_T} (ii)
\[
f_\rho(g.x) = \sum_{h \in Y_x} \rho(gh^{-1}) f(h.x),  \quad \textrm{for all $g$},
\]
and thus,  for every Borel set $B \subset G$,
\[
\int_B f_\rho(g.x) = \sum_{h \in Y_x} \Big( \int_G \chi_{g^{-1}B}(h) \, \rho(g) \, dm_G(g)\Big) \,  f(h.x).
\]
Suppose $\supp(\rho) \subset V$.  Then,  since $\rho$ and $f$ are non-negative,
\[
\chi_{B_V^{-}}(h) \leq  \int_G \chi_{g^{-1}B}(h) \, \rho(g) \, dm_G(g)  \leq \chi_{B_V^{+}}(h),  \quad \textrm{for all $h \in G$},
\]
whence 
\begin{equation}
\label{upperlower}
\sum_{h \in Y_x \cap B_V^{-}} f(h.x) \leq \int_B f_\rho(g.x) \,  dm_G(g) \leq \sum_{h \in Y_x \cap B_V^{+}} f(h.x),
\end{equation}
for all $x \in X$ and for every (pre-compact) Borel set $B \subset G$.  \\

Suppose $(G_t)$ is a convenient sequence of pre-compact Borel sets in $G$.  By \eqref{contain}, 
there exist sequences $(\delta_n)$ and $(t_n)$ such that $\delta_n \ra 0$,  and
\[
G_{t-\delta_n} \subset (G_t)_{V_n}^{-} \subset (G_t)_{V_n}^{+} \subset G_{t+\delta_n},  \quad \textrm{for all $t > t_n$}.
\]
Let $(\rho_n)$ be a sequence of non-negative continuous functions such that $\supp(\rho_n) \subset V_n$ and $m_G(\rho_n) = 1$ for all $n$.  
If we set $B = G_t$ and $\rho = \rho_n$ in \eqref{upperlower},  then
\[
\sum_{h \in Y_x \cap G_{t-\delta_n}} f(h.x) \leq \int_{G_t} f_{\rho_n}(g.x) \, dm_G(g) 
\leq \sum_{h \in Y_x \cap G_{t+\delta_n}} f(h.x)
\]
for all $x \in X$ and $t > t_n$.  Define
\[
\psi_{+}(x) = \varlimsup_{t \ra \infty} \frac{1}{m_G(G_t)} \sum_{h \in Y_x \cap G_t} f(h.x)
\qand
\psi_{-}(x) = \varliminf_{t \ra \infty} \frac{1}{m_G(G_t)} \sum_{h \in Y_x \cap G_t} f(h.x),
\]
for $x \in X$,  and set $E_n := E_{f_{\rho_n}}$ and $\psi_n := \overline{f}_{\rho_n}$.  
By \eqref{ratios}, 
\[
(1-\eps_n) \psi_{-}(x) \leq \psi_n(x)  \leq (1+\eps_n) \,  \psi_{+}(x)
\]
for all $x \in E_n$.  Furthermore,  for every $G$-ergodic $\mu \in \Prob(X)^G$, we have $\psi_n(x) = \mu(f_{\rho_n}) = \mu_Y(f)$ for all $x \in E_n(\mu) := E_{f_{\rho_n}}(\mu)$.  Since $\eps_n \ra \infty$,  we conclude that 
\[
\psi_{+}(x) = \psi_{-}(x) = \mu_Y(f),  \quad \textrm{for all $x \in E_f'(\mu) := \bigcap_n E_n(\mu)$}.
\]
By (ii),  each $E_n'(\mu)$ is $G$-invariant and $\mu$-conull,  and thus $E_f'(\mu)$ is $G$-invariant and $\mu$-conull as well.  By Lemma \ref{lemma_null},  $E_f := E_f'(\mu) \cap Y$ is a $\mu_Y$-conull Borel set in $Y$,  and 
\[
\lim_{t \ra \infty} \frac{1}{m_G(G_t)} \sum_{h \in Y_y \cap G_t} f(h.y) = \mu_Y(f),
\]
for all $y \in E_f$.
\end{proof}

\subsection{Positivity of lower densities}

We now give an application of Theorem \ref{Thm_Ergodic} to densities of generic return time sets. Given a convenient sequence $(G_t)$ of pre-compact Borel sets in $G$ we define
the associated \emph{lower density} $\Dens_{(G_t)}(A)$ of a subset $A \subset G$ by
\[
{\Dens}_{(G_t)}(A) :=  \varliminf_{t \ra \infty} \frac{|A \cap G_t|}{m_G(G_t)}.
\]
We now fix such a sequence once and for all and define the \emph{lower density function} $\Dens_Y$ by 
\begin{equation}
\label{def_DensY}
\Dens_{Y} : Y \ra [0,\infty], \quad \Dens_Y(y):={\Dens}_{(G_t)}(Y_y).
\end{equation}

\begin{theorem}
\label{Thm_Density}
For every $\mu \in \Prob(X)^G$,  $\Dens_Y(y) > 0$ for $\mu_Y$-almost every $y \in Y$.
\end{theorem}

\vspace{0.1cm}

\begin{proof}
Let $N := \{ y \in Y \,  : \,  \Dens_Y(y) = 0 \}$.  Fix a $G$-invariant Borel probability measure $\mu$ on $X$.  We want to show that $\mu_Y(N) = 0$.  By Corollary \ref{Cor_ErgodicDecomposition},  there is a Borel probability measure $\sigma$
on $\Prob(X)^G$,  supported on the set of $G$-ergodic Borel probability measures, 
such that
\[
\mu_Y(f) = \int_{\Prob(X)^G} \eta_Y(f) \,  d\sigma(\eta),
\]
for every bounded real-valued Borel function on $X$.  Hence it suffices to prove that 
$\Dens_Y(y) > 0$ for $\eta$-almost every $y$,  for every $G$-ergodic $\eta$.  To 
do this,  note that
\[
\Dens_Y(y) = \varliminf_{t \ra \infty} \frac{1}{m_G(G_t)} \sum_{h \in Y_y \cap G_t} 1.
\]
By Theorem \ref{Thm_Ergodic},  the right-hand side  converges to $\eta_Y(1)$ for $\eta$-almost every $y$,  and for every $G$-ergodic Borel probability measure $\eta$.  By Lemma \ref{Lemma_injectivemap} we have $\eta_Y(1) > 0$,  and we are done.
\end{proof}

\section{Transverse measures for actions of semidirect products}
\label{Sec:semidirect}

In this section we finally establish Theorem \ref{Stages} concerning restriction in stages for semi-direct products which is the main technical result of this article. We also provide several criteria for the intermediate cross-section to be separated; this will be used in the proof of Theorem \ref{ThmComm} and some of its variants.

\subsection{Basic setup}
Let $G$ be a lcsc group.  We assume that there are closed subgroups $N$ and $L$ of 
$G$,  with left Haar measures $m_N$ and $m_L$ respectively,   such that
\begin{itemize}
\item $G = NL$,  $N$ is normal in $G$ and $N \cap L = \{e\}$.  In other words,  $G$ is 
the semi-direct product of $N$ and $L$; 
\item the conjugation action of $L$ on $N$ preserves $m_N$.  
\end{itemize}
The first assumption in particular implies that every element $g$ in $G$ is of the form
$nl$ for \emph{unique} elements $n \in N$ and $l \in L$ and we write $n := \pr_N(g)$
and $l := \pr_L(g)$.  Note that $\pr_L : G \ra L$ is a continuous homomorphism. The second assumption implies that
\[
m_G(f) = \int_N \int_L f(nl) \,  dm_N(n) \, dm_L(l),  \quad f \in C_c(G),
\]
is a left Haar measure on $G$.  Note that 
\begin{equation}
\label{productmeasure}
m_G(W_N W_L) = m_N(W_N) m_L(W_L),  
\end{equation}
for all Borel sets $W_N \subset N$ and $W_L \subset L$. The key example that we have in mind is the following:
\begin{example}\label{SDPMainExample}
Let $G_o$ be a unimodular lcsc group, set $G := G_o^n$ and denote by  $\pi_k: G \to G_o$ the $k$th coordinate projection. For some $k$ let $N := N_k := \ker(\pi_k)$ and let $L := \Delta(G_o) < G$ be the diagonal subgroup. Then the triple $(G, N, L)$ satisfies all of the assumptions above. Indeed, since $G_o$ is unimodular,  $G$,  $N$ and $L$ are unimodular,  and the $L$-action on $N_k$ preserves the (bi-invariant) Haar measure on $N_k$. 
\end{example}
Given an open identity neighbourhood $U$ in $G$ and a $U$-transverse triple $(X, a, Z)$ over $G$ we make the following definitions: We fix Borel exhaustion $\cZ = \{Z_n \,  : \,  n \geq 1\}$ of $Z$,  and let $\cL = \{ L_n \,  : \,  n \geq 1 \}$ and $\cN = \{ N_n \,  : \,  n \geq 1 \}$ be fundamental exhaustions by compact sets of $L$ and $N$ respectively.  We define \[
Y := L.Z \qand Y_n := L_n.Z_n \qand \cY = \cZ_{\cL} \qand \cG = \{ N_nL_n \,  : \,  n \geq 1 \big\},
\]
and $a_L := a|_{L \times X}$ and $a_N := a|_{N \times X}$.  Note that $Y$ and $Y_n$
are Borel sets in $X$ by Lemma \ref{Lemma_BasicProp} (iv) and Lemma \ref{Lemma_Borel} (iii).  In particular,  $(Y,a_L)$ is a Borel $L$-space by Lemma \ref{Lemma_Borel} (i),  and $(Y, a_L, Z)$ is
a $U \cap L$-transverse triple over $L$. At the same time $Y$ is also a cross-section for the Borel action $a_N$ of $N$ on $X$. In the sequel we refer to it as the \emph{intermediate cross-section}.

\subsection{Criteria for separatedness of the intermediate cross-section}
We keep the notation of the previous subsection. In particular, $Z$ is a separated $G$-transversal in $X$ and at the same time a separated $N$-transversal in $Y$. The following example shows that $Y$ need not be separated when considered as an $N$-cross-section in $X$.

\begin{example}
\label{examplealpha}
Let $G = \bR^2$, $ N = \bR \times \{0\}$ and $L = \bR(\alpha,1)$,  where $\alpha$ is an irrational number,  and let $X = \bR^2/\bZ^2$,  $Z = \{(0,0)\}$ and $Y = L.Z$. 
Then,  for every $y = (t\alpha + \bZ, t + \bZ) \in Y$, 
\[
Y_y = \big\{ (m + n\alpha,0) \, : \,  m,n \in \bZ \big\},
\]
which clearly accumulates at $(0,0)$,  and thus $Y$ is not a separated cross-section.
\end{example}

Our next lemma provides a necessary and sufficient criterion for $Y$ to be a separated cross-section in terms of the return time set $\Lambda_a(Z)$.  However,  this criterion is
often difficult to verify,  so we also give two sufficient criteria,  which are easier to check in practise.  We use the following notation: if $W \subset N$,  then $W^l := l W l^{-1}$
for $l \in L$ and $W^L := \bigcup_{l \in L} W^l$.

\begin{lemma}
\label{Lemma_NcrossIFF}
We have $\Lambda_{a_N}(Y) = \pr_N(\Lambda_{a}(Z))^L$.  In particular,  $Y$ is a separated cross-section if and only if $\pr_N(\Lambda_{a}(Z))^L$ is uniformly discrete. This is the case if either
\begin{itemize}
\item[$(i)$] there is a compact set $K_L \subset L$ such that $K_L(\Lambda_a(Z) \cap L) = L$,  and $\pr_N(\Lambda_a(Z)^3)$ does not accumulate at the identity in $N$,  or 
\item[$(ii)$] $G$ is abelian,  and $\pr_N(\Lambda_a(Z))$ does not accumulate at the identity in $N$.  
\end{itemize}
\end{lemma}

\begin{proof}
Let us first show that $\Lambda_{a_N}(Y) = \pr_N(\Lambda_{a}(Z))^L$.  Consider the following equivalences: (1) $n \in \Lambda_{a_N}(Y)$,  i.e.  there exists $y \in Y$ such that $y' := n.y \in Y$.  Since $Y = L.Z$,  we can write $y = l_z.z$ and $y' = l_o.z_o$ for some 
$l_o,  l_z \in L$ and $z_o,  z \in Z$.  Hence (1) is equivalent to (2): $(l_o^{-1}n l_o)l_o^{-1}l_z \in Z_z$,  for \emph{some} $l_o,  l_z \in L$ and $z \in Z$.  Since $N$ is normal,  $N \cap L = \{e\}$,  $G = NL$,  and $z$ and $l_z$ are arbitrary,  (2) is equivalent to 
(3): $l_o^{-1}n l_o \in \pr_N(\Lambda_{a}(Z))$ for some $l_o \in L$,  or
equivalently,  $n \in \pr_N(\Lambda_a(Z))^L$. 

Note that if $G$ is abelian, then $\pr_N(\Lambda_{a}(Z))^L = \pr_N(\Lambda_{a}(Z))$, since the $L$-action on $N$ is trivial, and hence we deduce that Condition (ii) is sufficient. Now suppose that (i) holds and fix a compact subset $K_L \subset L$ such that $K_L(\Lambda_a(Z) \cap L) = L$.  Since $\eta \,  \Lambda_a(Z) \,  \eta^{-1} \subset \Lambda_a(Z)^3$ for 
all $\eta \in \Lambda_a(Z) \cap L$,  we have
\[
\pr_N(\Lambda_a(Z))^L 
\subset
\pr_N(\Lambda_a(Z)^3)^{K_L}.
\]
Suppose that $\pr_N(\Lambda_a(Z)^3)$ does not accumulate at the identity in $N$.   Then there is an identity neighbourhood $W$ in $N$ such that $\pr(\Lambda_a(Z)^3) \cap W = \{e\}$.  Let $U_N$ be an open identity neighbourhood in $N$ such that $k^{-1} U_N k \subset W$ for all $k \in K_L$.  Then,  
\begin{align*}
\pr_N(\Lambda_a(Z)^3)^{K_L} \cap U_N 
&\subseteq \big(\pr_N(\Lambda_a(Z)^3) \cap U_N^{K_L^{-1}}\big)^{K_L} \\[0.1cm]
&\subseteq \big(\pr_N(\Lambda_a(Z)^3) \cap W \big)^{K_L} = \{e\},
\end{align*}
and thus $\pr_N(\Lambda_a(Z))^L \cap U_N = \{e\}$ as well.  In particular,  $Y$ is a $U_N$-separated cross-section.
\end{proof}

One might suspect that the cocompactness assumption in Lemma \ref{Lemma_NcrossIFF} (i) which assumes that there is a compact set $K_L$ in $L$ such that $K_L(\Lambda_a(Z) \cap L) = L$ is superfluous, since it has no counterpart in the abelian case. Our next example shows that it cannot simply be dropped in the non-abelian case.

\begin{example}
Let $G = \SL_2(\bR) \times \SL_2(\bR)$,  let $L$ denote the diagonal in $G$,  and let $N = \SL_2(\bR) \times \{e\}$.  Note that $\pr_N(g_1,g_2) = g_1g_2^{-1}$.  Let $\Gamma_o = \SL_2(\bZ)$ and $\Gamma = \Gamma_o \times \Gamma_o$,  and set
\[
X = G/\Gamma \qand Z = \{\Gamma\},
\]
where $G$ acts on $X$ by left multiplication.  Then it is easy to check that 
$\Lambda_a(Z) = \Gamma$ and $ \pr_N(\Lambda_a(Z))^L = \Gamma_o^{\SL_2(\bR)} \times \{e\}$.  Since $\Gamma_o$ contains non-trivial unipotent elements,  
$ \Gamma_o^{\SL_2(\bR)}$ accumulates at $e$ in $\SL_2(\bR)$,  and thus $Y := L.Z =\big\{ (g_o\Gamma_o,g_o\Gamma_o) \,  : \,  g_o \in \SL_2(\bR)\big\}$ is \emph{not} a separated cross-section for the $N$-action.  However,  since $\pr_N(\Lambda_a(Z)^3) = \Gamma_o \times \{e\}$,  the second condition in Lemma \ref{Lemma_NcrossIFF} (i) clearly holds.
\end{example}

\subsection{Restriction in stages}
In this section we establish Theorem \ref{Stages} from the introduction. In fact, we are going to establish a more general version which also works for certain infinite measures. We keep the notation of the previous subsection and assume in addition that $Y$ is a $U_N$-separated cross-section for 
the $N$-action on $X$ for some open identity neighbourhood $U_N$ in $N$. 

Let now $\mu \in M(X, \cZ_{\cG})^G$ and let $\mu_Z := \res G X Y(\mu) \in M(Z, \cZ)^{R_{G,Z}}$. Since $\mu \in M(\cY_{\cN})^L$ (as $L \subset G$ and $\cZ_\cG = (\cZ_{\cL})_{\cN} = \cY_{\cN}$) we can also form the intermediate measure $\mu_Y := \res N X Y(\mu) \in M(Y, \cY)^{R_{N, Y}}$.

\begin{theorem}\label{StagesInfinite} Assume that the $N$-cross-section $Y$ is separated and let $\mu \in M(C, \cZ_{\cG})^G$ and $\mu_Z := \res G X Y(\mu) \in M(Z, \cZ)^{R_{G,Z}}$.
\begin{enumerate}[(i)]
\item The measure $\mu_Y \in M(Y, \cY)^{R_{N, Y}}$ is $L$-invariant and hence the transverse measure $(\mu_Y)_Z := \res L Y Z(\mu_Y)$ is well-defined.
\item $(\mu_Y)_Z = \mu_Z$.
\end{enumerate} 
\end{theorem}
\begin{proof} Fix identity neighbourhoods $V, V_N$ and $V_L$ in $G,  N$ and $L$ respectively,  such that 
\[
V^{-1}V \subset U \qand V_N^{-1}V_N \subset U_N \qand V_L^{-1}V_L \subset L \cap U.
\]
We may without loss of generality assume that $V_N V_L \subset V$.  Note that 
$V \times Z$,  $V_L \times Z$ and $V_N \times Y$ are $a$-injective,  $a_L$-injective
and $a_N$-injective Borel sets respectively (cf. Lemma \ref{Lemma_BasicProp} (i)). \\

To prove that $\mu_Y$ is $L$-invariant,  
we fix $l \in L$,  a Borel set $B \subset Y$ and an identity neighbourhood $V_l \subset V_N$ such that $l^{-1}V_l l \subset V_N$.  Then,  since $l.B \subset Y$ and $V_l \times l.B$ is an injective Borel set in $N \times Y$,   
\[
\mu_Y(l.B) = \frac{\mu(V_l l.B)}{m_N(V_l)} = \frac{\mu(l^{-1}V_l l.B)}{m_N(V_l)} = \frac{m_N(l^{-1}V_l l)}{m_N(V_l)} \,  \mu_Y(B) = \mu_Y(B),
\]
where we in the second equality have used that $\mu$ is $L$-invariant,  and in the last identity that the conjugation action of $L$ on $N$ preserves $m_N$.  Since $l$ and $B$ are arbitrary,  $\mu_Y$ is $L$-invariant. \\

Let us now prove the identity $(\mu_Y)_Z = \mu_Z$.  Note that
\[
\mu(V_N V_L. B) = m_N(V_N) \mu_Y(V_L.B) = m_N(V_N)m_L(V_L) (\mu_Y)_Z(B),
\]
for every Borel set $B \subset Z$,  where we in the first identity have used that 
$\mu_Y$ is the $Y$-transverse of $\mu$ for the $N$-action,  and in the second identity that $(\mu_Y)_Z$ is the transverse measure of $\mu_Y$ for the $L$-action on $Y$.  On the other hand,  
\[
\mu(V_N V_L. B) = m_G(V_N V_L) \mu_Z(B) = m_N(V_N)m_L(V_L) \mu_Z(B),
\]
since $\mu_Z$ is the $Z$-transverse measure of $\mu$ for the $G$-action on $X$.  Since $B$ is arbitrary,  we conclude that $(\mu_Y)_Z = \mu_Z$. 
\end{proof}
\begin{proof}[Proof of Theorem \ref{Stages}] This is just a special case of Theorem \ref{StagesInfinite} in which $\mu$ is finite.
\end{proof}
There are several equivalent ways to formulate Theorem \ref{StagesInfinite}, for example:
\begin{corollary} Assume that the $N$-cross-section $Y$ is separated Then the composition $\res{L}{Y}{Z} \circ \res{N}{X}{Y}$ is well-defined on $M(X, \cZ_{\cG})^G$  and satisfies \[\res{L}{Y}{Z} \circ \res{N}{X}{Y} = \res{G}{X}{Z}: M(X, \cZ_{\cG})^G \to M(Z, \cZ)^{R_{G,Z}}.\]
Similarly, the composition $\ind NXY \circ \ind LYZ$ is well-defined on $M(Z, \cZ)^{R_{G,Z}}$ and satisfies
\[
\ind NXY \circ \ind LYZ = \ind GXZ: M(Z, \cZ)^{R_{G,Z}} \to M(X, \cZ_{\cG})^G.
\]
\end{corollary}
Indeed, the first version is just a reformulation, and the second version follows from Lemma \ref{Lemma_injectivemap}. Yet another way to express the same conclusion is to say that for $\mu \in M(X, \cZ_{\cG})^G$ we have
\begin{equation}\label{indresFormula}
\res N X Y(\mu)= \ind L Y Z \left( \res G X Y(\mu) \right).
\end{equation}
Note that the group $N$ does not appear on the right-hand side. This implies the following independence of the measure $\mu_Y$ from the choice of semi-direct splitting:
\begin{corollary}
\label{Cor_Ncrossmeasures}
Let $\mu \in M(X;\cZ_{\cG})^G$ and suppose that $N_1$ and $N_2$ are closed and normal unimodular subgroups of $G$ such that
\[
N_1 \cap L = N_2 \cap L = \{e\} \qand G = N_1 L = N_2 L.
\]
Assume that $Y$ is a separated cross-section for both $a_{N_1}$ and $a_{N_2}$ and that the $L$-actions on $N^{(1)}$ and $N^{(2)}$ by conjugation preserve the respective Haar-measures. Then
\[
\pushQED{\qed}
\res {N_1} X Y(\mu) = \res {N_2} X Y(\mu) = \ind L Y Z \left( \res G X Y(\mu) \right).\qedhere\popQED
\]
\end{corollary}

\subsection{Compatibility of transverse measures}

We keep the notation of the previous subsection, including the assumption that $Y$ is a $U_N$-separated cross-section for 
the $N$-action on $X$ for some open identity neighbourhood $U_N$ in $N$. Our goal here is to establish the following compatibility theorem about transverse measures which will be used in the proof of Theorem \ref{Joining}  to show that the normalized intersection measure is a joining.

\begin{theorem}\label{Compatibility} Let $(T,b_o)$ be a Borel $L$-space and $\pi : X \ra T$ be a Borel $G$-map,  where $G$ acts on $T$ via the homomorphism $\pr_L$.  Suppose $\mu \in M(X;\cZ_{\cG})^G$ is finite and $\pi_*\mu$ is $L$-ergodic.  Then \[\pi_*\mu_Y = \mu_Y(Y) \,  \pi_*\mu.\]
\end{theorem}
\begin{remark}
We do not know whether $L$-ergodicity of $\pi_*\mu$ is necessary.  In the proof below,  we show that $\pi_*\mu_Y$ is always absolutely continuous with respect to 
$\pi_*\mu$,  but it seems to be a difficult problem to explicate the Radon-Nikodym derivative in general.
\end{remark}

\begin{proof} Since $\mu$ is finite,  $\mu_Y$ is finite as well,  so $\pi_*\mu$ and $\pi_*\mu_Y$ are both finite $L$-invariant Borel measures on $T$.  Since 
$\pi_*\mu$ is assumed to be $L$-ergodic,  to prove the theorem,  it is enough to show that $\pi_*\mu_Y$ is absolutely continuous with respect to $\pi_*\mu$.  To do this,  let $B \subset T$ be a Borel set such that $\pi_*\mu(B) = \mu(\pi^{-1}(B)) = 0$.  We want to prove that $
\pi_*\mu_Y(B) = \mu_Y(\pi^{-1}(B) \cap Y) = 0$.
Since $N$ acts trivially on $T$,  the pre-image $\pi^{-1}(B)$ is an $N$-invariant Borel 
set in $X$.  Hence,  since $\mu_Y$ is the $Y$-transverse measure of $\mu$ for the $N$-action on $X$,  
\[
m_N(V_N) \mu_Y(\pi^{-1}(B) \cap Y) = \mu(V_N.(\pi^{-1}(B) \cap Y)) \leq \mu(V_N.\pi^{-1}(B)) =  \mu(\pi^{-1}(B)) = 0,
\]
and thus $\pi_*\mu_Y(B) = 0$.
\end{proof}

\section{Intersection spaces}\label{SecIntersectionSpaces}

In this section we apply the general theory of transverse measures for semidirect products developed in the previous section to the special case where $G$ is as in Example \ref{SDPMainExample} and $X$ is a product space. In this case, the intermediate cross-section $Y$ is precisely the intersection space discussed in the introduction, and we can use this to deduce Theorem \ref{Joining} from Theorem \ref{StagesInfinite}. Once the good properties of the intersection measure are established, we obtain Theorem \ref{MTR3} and Theorem \ref{PD3} (and hence their special case Theorem \ref{MTR2} and Theorem \ref{PD2}) by applying the results of the second part. We also establish the commensurability criterion from Theorem \ref{ThmComm}.

\subsection{Basic setup}\label{SetupIntersectionSpaces}

Let $G_o$ be a lcsc unimodular group with left Haar measure $m_{G_o}$. As in Example \ref{SDPMainExample} we define $G := G_o^r$, $N_k := \ker(\pi_k)$ (where $\pi_k: G \to G_o$ is the $k$th coordinate projection) and $L := \Delta(G_o)$. Furthermore, we assume that we are given $G_o$-spaces $(X_1,a_1),\ldots,(X_r,a_r)$ and corresponding separated cross-sections $Z_1, \dots, Z_r$ for some identity neighbourhood $U_k$ in $G_o$. We also pick open identity neighbourhoods $U_k$ in $G_o$ such that $Z_k$ is $U_k$-separated and abbreviate $\Lambda_k := \Lambda_{a_k}(Z_k) \subset G_o$. We now set
\[
 X:= X_1 \times \ldots \times X_r, \quad Z:= Z_1 \times \dots \times Z_r, \quad U := U_1 \times \dots \times U_r \qand \Lambda := \Lambda_1 \times \dots \times \Lambda_r.
\]
and denote by $a$ the product of the $G_o$-actions $a_1,\ldots,a_r$ so that $(X,a)$ is a Borel $G$-space and $Z$ is a $U$-separated cross-section with return time set $\Lambda_a(Z) = \Lambda$.
We now apply the theory of the previous section to transverse triple $(X,a, Z$ of the semidirect product $G = N_kL$ on $X$. The key observation, which links the theory developed in the last section to the results presented in the introduction, is that the intermediate transversal $Y := L.Z$ is precisely the intersection space
\begin{eqnarray*}
Y \quad = \quad X^{[r]} & := & \Big\{ (x_1,\ldots,x_r)\in X^r  \,  : \,  \bigcap_{k=1}^r (Z_k)_{x_k} \neq \emptyset \,  \Big\}\\
 &=&  \Big\{ (x_1,\ldots,x_r) \in X^r  \,  : \,  \exists\, g_o \in G_o\; \forall\, k \in \{1, \dots, r\}:\; g_o.x_k \in Z_k \Big\} 
\end{eqnarray*}
considered in the introduction. It thus follows from the results of the previous section that $X^{[r]}$ is Borel and $N_k.X^{[r]} = X$ for every $k \in \{1, \dots, r\}$.

To study finiteness properties of measures on the spaces $X$, $Y$ and $Z$ we fix a fundamental exhaustion $\cG_o = \{ K^{(o)}_n \, : \,  n \geq 1\}$ of $G_o$ by compact sets and a Borel exhaustion $\cZ_k = \{Z^{(k)}_n \,  : \,  n \geq 1\}$ of $Z_k$ for each $k \in \{1, \dots, r\}$.  Let $\cG$ and $\cZ$ denote the fundamental exhaustion of $G$ by compact sets and the Borel exhaustion of $Z$ given by 
\[
\cG = \{ K_n^{(o)} \times \ldots \times K_n^{(o)} \,  : \,  n \geq 1 \} 
\qand \cZ = \{ Z_n^{(1)} \times \cdots \times Z_n^{(r)} \,  : \,  n \geq 1\}.
\]
We denote by $\cL$ and $\cN_k$ the restrictions of $\cG$ to $L$ and $N_k$ respectively,  and we denote by $\cX^{[r]}$ the $\cL$-suspension of $\cZ$,  and by 
$\cX^{[r]}_{\cN_k}$ the $\cN_k$-suspension of $\cX^{[r]}$.

\subsection{Commensurability of cross-sections}
We keep the notation of the previous subsection. In order for our general theory to apply, we need to ensure that the intermediate cross-section $Y = X^{[r]}$ is separated for each of the actions $a|_{N_k}$ of $N_k$. We recall from the introduction that the separated cross-sections $Z_1, \dots, Z_k$ are called \emph{commensurable} if this is the case. From Lemma \ref{Lemma_NcrossIFF} we can derive the following criteria for commensurability.
\begin{lemma}
\label{Lemma_YcrossINT}
For a fixed $k$,  $Y = X^{[r]}$ is a separated cross-section for $a|_{N_k}$ if either
\vspace{0.1cm}
\begin{itemize}
\item[$(i)$] there exists a compact set $K_o \subset G_o$ such that 
$K_o \Big(\bigcap_{j=1}^r\Lambda_j \Big) = G_o$,  and for all $j \neq k$,
$\Lambda_{j}^3 \Lambda_{k}^3$ does not accumulate at the identity 
in $G_o$,  or \vspace{0.1cm}
\item[$(ii)$] $G_o$ is abelian,  and for all $j \neq k$,  $\Lambda_{j} \Lambda_{k}$ does not accumulate at the identity in $G_o$. 
\end{itemize}
In particular, the cross-sections $Z_1, \dots, Z_k$ are commensurable if one of these conditions holds for all $k \in \{1,\dots, r\}$.
\end{lemma}
\begin{remark}
Before we embark on the proof,  let us first make a few preliminary observations. Given a subset $A \subset G_o$ we set $\Delta_r(A) = \big\{ (a,\ldots,a) \,  : \,  a \in A \big\}$. 
Since $\Lambda_a(Z) = \Lambda$ we then have
\[
\Lambda_a(Z) \cap L  = \Delta_r\Big( \bigcap_{j=1}^r \Lambda_j\Big) 
\qand
\Lambda_a(Z)^p = \Lambda_1^p \times \ldots \times \Lambda_r^p,
\]
for every positive integer $p$. Furthermore,  
\[
\pr_{N_k}(g) = (g_1 g_k^{-1},\ldots,e,\ldots, g_r g_k^{-1}) \in N_k,  
\]
for all $g = (g_1,\ldots,g_r) \in G$,  where $e$ is in the $k$'th position.  In particular,
\[
\pr_N(\Lambda_a(Z)^p) \subset \Lambda^p_1 \Lambda^p_k \times \cdots \times \{e\} \times \cdots \times \Lambda^p_r \Lambda^p_k.
\]
\end{remark}
\begin{proof}[Proof of Lemma \ref{Lemma_YcrossINT}]
(i) By Lemma \ref{Lemma_NcrossIFF} (i),  $Y$ is a separated cross-section for $a|_{N_k}$ if there is a compact set $K_L \subset L$ such that $K_L \pr_{N_k}(\Lambda_a(Z) \cap L) = L$ and $\pr_{N_k}(\Lambda_a(Z)^3)$ does not accumulate at the identity in $N_k$.  From the remarks above,  the first condition just means that there 
is a compact set $K_o \subset G_o$ such that 
\[
\Delta_r(K_o) \Delta_r\Big( \bigcap_{j=1}^r \Lambda_j\Big) = \Delta_r(K_o  \Big(\bigcap_{j=1}^r \Lambda_j\Big) \Big) = L,
\]
while the second condition means that $\Lambda_j^3 \Lambda_k^3$ does not accumulate at zero in $G_o$ for every $j \neq k$.  The same argument applies to (ii),  but now using Lemma \ref{Lemma_NcrossIFF} (ii). 
\end{proof}
\begin{proof}[Proof of Theorem \ref{ThmComm}] This is just a special case of Lemma \ref{Lemma_YcrossINT}.
\end{proof}

\subsection{The intersection measure}
We keep the notation of Section \ref{SetupIntersectionSpaces}. Moreover, we assume that the separated cross-sections $Z_1, \dots, Z_r$ are commensurable. We now assume that we are given $G_o$-invariant measures $\mu_1, \dots, \mu_r$ on $X_1, \dots, X_r$ respectively such that $\mu_k$ is $(\cZ_k)_{G_o}$-finite. Then, by definition, we have
\[
\mu := \mu_1 \otimes \dots \otimes \mu_r \in M(X, \cZ_\cG)^G \subset M(X, \cX^{[r]}_{\cN_k})^{N_k}
\]
for all $k\in \{1, \dots, r\}$ and hence for any such $k$ we can form the tranverse measure
\[
\mu^{[r]}_k := \res{N_k}{X}{X^{[r]}} \in M(X^{[r]}, \cX^{[r]}).
\]
By Corollary \ref{Cor_Ncrossmeasures} the measure $\mu^{[r]} := \mu^{[r]}_1 = \dots = \mu^{[r]}_r$ is actually independent of the choice of $k$.

\begin{definition} The measure $\mu^{[r]} \in M(X^{[r]}, \cX^{[r]})$ is called the \emph{intersection measure} of $\mu_1, \dots, \mu_r$.
\end{definition}

From the general theory we infer the following properties of intersection measures:
\begin{theorem}
\label{Thm_IntersectionMeasures}
Suppose that $Z_1,\ldots,Z_r$ are commensurable and for every $k \in \{1, \dots, r\}$ let $\mu_k$ be a $(\cZ_k)_{G_o}$-finite $G_o$-invariant Borel measure on $X_k$. Then the intersection measure $\mu^{[r]}$ of $\mu_1, \dots, \mu_r$ has the following properties:
\vspace{0.1cm}
\begin{itemize}
\item[$(i)$] $\mu^{[r]} \in M(X^{[r]};\cX^{[r]})^L$ is $L$-invariant. \vspace{0.1cm}
\item[$(ii)$] $\mu^{[r]}$ is the unique measure in $M(X^{[r]};\cX^{[r]})^L$ such that $(\mu^{[r]})_{Z} = (\mu_1)_{Z_1} \otimes \dots \otimes (\mu_k)_{Z_k}$.\vspace{0.1cm}
\item[$(iii)$] If $\mu_1,\ldots,\mu_r$ are finite,  then $\mu^{[r]}$ is finite.  \vspace{0.1cm}
\item[$(iv)$] If $\mu_1,\ldots,\mu_r$ are $G_o$-ergodic Borel probability measures,  then $\mu^{[r]}/\mu^{[r]}(X^{[r]})$ is a $G_o$-invariant joining of $\mu_1,\ldots,\mu_r$.
\end{itemize}
\end{theorem}

\begin{remark}
We stress that $\mu^{[r]}$ does not need to be $L$-ergodic,  even if $\mu_1,\ldots,\mu_r$ are all $G_o$-ergodic.  
\end{remark}

\begin{proof} (i) and (ii) follow from Theorem \ref{StagesInfinite} applied to $\mu = \mu_1 \otimes \dots \otimes \mu_r$ and the fact that $\res{L}{X^{[r]}}{Z}$ is injective (Lemma \ref{Lemma_injectivemap}). (iii) If $\mu_1,\ldots,\mu_r$,  then $\mu$ is finite, 
and thus $\mu^{[r]} = \mu_Y$ is finite as well by Remark \ref{Remark_muYfinite}.  (iv) Fix $k \in \{1, \dots, r\}$\ and denote by $p_k: X \to X_k$ the projection onto the $k$th factor. We have to show that $(p_k)_*(\mu_k^{[r]}/\mu_k^{[r]}(X^{[r]})) = \mu_k$. Since  $p_k$ is $L$-equivariant (where $L$ acts on $X_k$ via the $k$'th coordinate), this follows from Theorem \ref{Compatibility} applied to the map $p_k$.
\end{proof}

\begin{proof}[Proof of Theorem \ref{Joining}] This is just a special case of Theorem \ref{Thm_IntersectionMeasures} in which $\mu_1, \dots, \mu_r$ are finite.
\end{proof}
\begin{proof}[Proof of Theorem \ref{Intro1}] This is just a special case of Theorem \ref{Thm_IntersectionMeasures} in which $X_1 = \dots = X_r$ and $\mu_1 =  \dots =  \mu_r$ is finite.
\end{proof}

\subsection{Applications} 
We now derive Theorem \ref{MTR3} and Theorem \ref{PD3} from the introduction from Theorem \ref{Thm_IntersectionMeasures}. We thus assume that that $G_o$ is non-compact,  $X_1,\ldots,X_r$ are Polish spaces and $a_1,\ldots,a_r$ are jointly continuous $G_o$-actions. We also assume that $\mu_1, \dots, \mu_r$ are Borel probability measures on $X_1, \dots X_r$ respectively. Theorem \ref{Thm_IntersectionMeasures} (i)--(iii) then implies that their intersection measure $\mu^{[r]}$ is a finite $L$-invariant Borel measure on $X^{[r]}$ such that $(\mu^{[r]})_Z = (\mu_1)_{Z_1} \otimes \dots \otimes (\mu_r)_{Z_r}$. We observe that for every $z = (z_1,\ldots,z_r) \in Z$,  
\begin{equation}
\label{Zz}
Z_z = \big\{ g \in G \,  : \,  g.(z_1,\ldots,z_r) \in Z_1 \times \cdots \times Z_r \big\} = \bigcap_{k=1}^r (Z_k)_{z_k}.
\end{equation}
\begin{proof}[Proof of Theorem \ref{MTR3}] By Theorem \ref{Thm_Poincare},  there is a $(\mu^{[r]})_Z$-conull Borel set $Z' \subset Z$ such that for all $z \in Z'$,  there is a sequence $g_n \in Z'_{z}$ such that
$g_n \ra \infty$ and  $g_n.z \ra z$.  If we unwrap this using \eqref{Zz} and the formula 
$(\mu^{[r]})_Z = (\mu_1)_{Z_1} \otimes \dots \otimes (\mu_r)_{Z_r}$,  then we see that we are done.
\end{proof}
\begin{proof}[Proof of Theorem \ref{PD3}] By Theorem \ref{Thm_Density}, applied to the $L$-action on $X^{[r]}$, we have
\[
\varliminf_{t \ra \infty} \frac{|Z_z \cap G_t|}{m_G(G_t)} > 0, \quad \textrm{for $(\mu^{[r]})_Z$-almost every $z \in Z$.  }
\]
In view of \eqref{Zz} and since $(\mu^{[r]})_Z = (\mu_1)_{Z_1} \otimes \dots \otimes (\mu_r)_{Z_r}$, this is what we want to prove.
\end{proof}

\begin{proof}[Proofs of Theorem \ref{MTR2} and \ref{PD2}] These are just the special cases $X_1 = \dots = X_r$ and $\mu_1 = \dots = \mu_r$ of Theorem \ref{MTR3} and \ref{PD3} respectively.
\end{proof}

\section{Application to uniform approximate lattices}\label{SecAUL}
We now specialize the results of the previous section to our main case of interest and derive Theorem \ref{MTR1}, Theorem \ref{PD1} and Theorem \ref{IntroTwisted} from the introduction.  In fact, we will establish a slightly more general result concerning transverse triples whose return time sets are uniform approximate lattices.

\subsection{Uniform approximate lattices and commensurability of cross-sections}

Let $G_o$ be a unimodular lcsc group A subset $\Theta \subset G_o$ is called \emph{symmetric} if $\Theta = \Theta^{-1}$. A symmetric subset $\Theta \subset G_o$ is called \emph{cocompact} if there is a compact set $K \subset G_o$ such that $K\Theta = G_o$. If $\Theta$ is a symmetric subset of $G$ which contains $e$, then we denote by $\Theta^\infty$ the subgroup of $G$ generated by $\Theta$,  i.e.  $\Theta^\infty = \bigcup_{p \geq 1} \Theta^p$. \\

We now consider the following situation: Let $(X_o, a_o, Z_o)$ be a $U_o$-transverse triple for $G_o$ for some identity neighbourhood $U_o$ in $G_o$ and denote by $\Lambda_o := \Lambda_{a_o}(Z_o) \subset G_o$ the associated return time set. For every $\lambda \in \Lambda_o^\infty$ the set $\lambda.Z_o$ is then a cross-section for $(X,a)$ and since $(Z_k)_{\lambda_k.z_o} = \lambda_k (Z_o)_{z_o} \lambda_k^{-1}$ for all $z_o \in Z_o$ we see that $\lambda.Z$ is a $U_o^\lambda$-separated cross-section and $\Lambda_{a_o}(\lambda.Z) = \Lambda_o^\lambda$. Here we are interested in the following problem: Given $\lambda_1, \dots, \lambda_r \in \Lambda^\infty$, are the cross-sections $\lambda_1.Z_o, \dots, \lambda_r.Z_o$ commensurable so that our general theory applies? According to Lemma \ref{Lemma_YcrossINT} it suffices to ensure that the subsets $\Lambda_k := \Lambda_{a_o}(\lambda_k.Z) = \Lambda_o^{\lambda_k}$ are cocompact (or equivalently, that $\Lambda_o$ is cocompact) and that the sets $\Lambda_j^3\Lambda_k^3$ do not accumulate at $e$ for any $j\neq k$. Since $\lambda_1, \dots, \lambda_r \in \Lambda_o^\infty$ we can find a positive integer $p_r$ such that
$\lambda_k \in \Lambda^{p_r}_o$ for all $k$. We then have
\[
\Lambda_j^p \Lambda^p_k \subset \Lambda_o^{2(p+2p_r)},  \quad \textrm{for all $p \geq 1$}.
\]
Thus $\lambda_1 Z, \dots, \lambda_r Z$ are commensurable provided $\Lambda_o$ is cocompact and $\Lambda_o^{p+2p_r}$ does not accumulate at the identity.\\

We now recall from the introduction that a symmetric subset $\Theta$ of a lcsc group $G_o$ is called and \emph{approximate subgroup} if $e\in \Theta$ and there is a finite set $F \subset G_o$ such that $\Theta^2 \subset \Theta F$. We say that $\Theta$ is \emph{cocompact} if there is a compact set $K \subset G_o$ such that $K\Theta = G_o$.  Note that if an approximate subgroup $\Theta$ is a uniformly discrete subset of $G_o$,  then so is every iterated product set $\Theta^p$ for $p \geq 1$.   Conversely,  if $\Theta$ is cocompact and $\Theta^p$ is uniformly discrete for all $p$,  then $\Theta$ is an approximate subgroup of $G_o$ (\cite[Proposition 2.9]{BjorklundHartnick1}). In this case we call $\Theta$ a \emph{uniform approximate lattice} in $G_o$. From the discussion above we thus infer:

\begin{lemma}
\label{Lemma_Approx}
If $\Lambda_o$ is a uniform approximate lattice in $G_o$, then for all $\lambda_1, \dots, \lambda_r \in \Lambda_o^\infty$ the separated cross-sections $\lambda_1 Z,\ldots, \lambda_r Z$ are commensurable.
\end{lemma}

\subsection{Twisted multiple recurrence and twisted positive density}
Let $G_o$ be a lcsc group and let $(X,a)$ be a Polish $G$-space such that $a_o$ is jointly continuous. Moreover, let $Z_o \subset G_o$ be a $U$-separated cross-section for some open identity neighbourhood $U_o \subset G_o$ with return time set $\Lambda_o := \Lambda_{a_o}(Z_o)$.

\begin{theorem}
\label{Thm_Commensurable}
Suppose $\Lambda_o$ is a uniform approximate lattice in $G_o$ and let $\mu_o$ be a $G$-invariant Borel probability measure on $X_o$.
\begin{itemize}
\item[$(i)$] There is a $(\mu_o)_{Z_o}^{\otimes r}$-conull Borel set $Z' \subset Z_o^r$ such 
that for all $(z_1,\ldots,z_r) \in Z'$,  there is a sequence $g_n \in \bigcap_{k=1}^r \lambda_k (Z_o)_{z_k} \lambda_k^{-1}$ such that
\[
g_n \ra \infty \qand \lambda_k^{-1}g_n\lambda_k.z_k \ra z_k,  \quad \textrm{for all $k$},
\]
as $n \ra \infty$.  \vspace{0.1cm}
\item[$(ii)$] For every convenient sequence $(G_t)$ of pre-compact Borel sets in $G_o$, 
\[
\varliminf_{t \ra \infty} \frac{\big| \big( \bigcap_{k=1}^r \lambda_k (Z_o)_{z_k} \lambda_k^{-1}\big) \cap G_t \big|}{m_G(G_t)} > 0,
\]
for $(\mu_o)_{Z_o}^{\otimes r}$-almost every $(z_1,\ldots,z_r) \in Z_o^r$.
\end{itemize}
\end{theorem}

\vspace{0.1cm}

\begin{proof} We abbreviate $Z_k := \lambda_k.Z_o$ and $\Lambda_k := \Lambda_o^{\lambda_k}$ for all $k \in \{1, \dots, r\}$. Since $\Lambda_o$ is a uniform approximate lattice in $G_o$, the cross-sections $Z_1, \dots, Z_r$ are commensurable by Lemma \ref{Lemma_Approx}. Note that $Z := Z_1 \times \dots \times Z_r = (\lambda_1,\ldots,\lambda_r).Z_o^r$.  We set $\mu_k := \mu_o$ for $k=1,\ldots,r$,  and note that
\begin{equation}
\label{mukZk}
(\mu_k)_{Z_k} = (\lambda_k)_*(\mu_o)_{Z_o},  \quad \textrm{for all $k$},
\end{equation}
by Lemma \ref{Lemma_YgY}.  Moreover,  $(Z_k)_{\lambda_k.z_k} = \lambda_k (Z_o)_{z_k} \lambda_k^{-1}$ for all $z_k \in Z_o$.  \\

(i) By Theorem \ref{MTR3},  there is a $\prod_{k=1}^r (\mu_k)_{Z_k}$-conull Borel set $Z'' \subset Z$ such that for all $r$-tuples $(z'_1,\ldots,z'_r) \in Z''$,  there is a sequence $g_n \in \bigcap_{k=1}^r (Z_k)_{z'_k}$ such that
\[
g_n \ra \infty \qand g_n.(z'_1,\ldots,z'_r) \ra (z'_1,\ldots,z'_r),
\]
as $n \ra \infty$.  Set $Z' := (\lambda_1,\ldots,\lambda_r)^{-1}.Z'' \subset Z_o^r$.
By \eqref{mukZk},  $Z'$ is a $(\mu_o)_{Z_o}^{\otimes r}$-conull Borel subset of $Z_o^r$.
Furthermore,  for all $(z_1,\ldots,z_r) \in Z'$,  
\[
(z'_1,\ldots,z'_r) = (\lambda_1.z_1,\ldots,\lambda_r.z_r) \in Z'',
\]
and thus there is a sequence $g_n \in \bigcap_{k=1}^r (Z_k)_{z'_k} = \bigcap_{k=1}^r \lambda_k (Z_o)_{z_k} \lambda_k^{-1}$,  such that
\[
g_n \ra \infty \qand g_n\lambda_k.z_k \ra \lambda_k.z_k,
\]
as $n \ra \infty$. \\

(ii) By Theorem \ref{PD3},  for every convenient sequence $(G_t)$
of pre-compact Borel sets in $G_o$,  
\[
\varliminf_{t \ra \infty} \frac{\big| \big(\bigcap_{k=1}^r (Z_k)_{z'_k}\big) \cap G_t\big|}{m_G(G_t)} > 0,
\]
for all $(z'_1,\ldots,z'_r)$ in some $\prod_{k=1}^r (\mu_k)_{Z_k}$-conull subset $Z'' \subset Z$.  We thus see that 
\[
\varliminf_{t \ra \infty} \frac{\big| \big(\bigcap_{k=1}^r \lambda_k (Z_o)_{z_k} \lambda_k^{-1} \big) \cap G_t\big|}{m_G(G_t)} > 0,
\]
for all $(z_1,\ldots,z_r) \in Z' := (\lambda_1,\ldots,\lambda_r)^{-1}.Z''$,  and $Z'$ is a 
$(\mu_o)_{Z_o}^{\otimes r}$-conull subset of $Z_o^r$ by \eqref{mukZk}.
\end{proof}
To close the circle, we now finally return to the setting of Subsection \ref{SecMotivation} which motivated this whole article. Thus let $P_o \subset G_o$ be a uniformly discrete subset such that $\Omega^\times_{P_o}$ admits a $G_o$-invariant probability measure $\mu_o$ and such that $\Lambda_o := P_oP_o^{-1}$ is a uniform approximate lattice. Denote by $a_r$ the right-multiplication action of $G_o$ on $\Omega^\times_{P_o}$. Then
\begin{itemize}
\item $X := \Omega^\times_{P_o}$ is locally compact, hence Polish, and the $G$-action $a_r$ is jointly continuous;
\item $Z := \mathcal T_{P_o}$ is a cross-section for $(X, a_r)$;
\item the return time set of $(X, a_r, Z)$ is $\Lambda_o$. Indeed, as in the proof of Proposition \ref{UDExamples} one sees that $\Lambda_o \subset \Lambda_{a_r}(Z) \subset \overline{\Lambda_o}$, but $\Lambda_o$ is unifomly discrete, hence closed. 
\end{itemize}
In particular, $(X, a_r, Z)$ is a transverse triple over $G_o$ whose return time set is the uniform approximate lattice $\Lambda_o$.
\begin{proof}[Proof of Corollary \ref{IntroTwisted}] Apply Theorem \ref{Thm_Commensurable} to the triple $(X, a_r, Z) := (\Omega^\times_{P_o}, a_r, \cT_{P_o})$.
\end{proof}
\begin{proof}[Proofs of Theorem \ref{MTR1} and Theorem \ref{PD1}] These are just the special cases of Corollary \ref{IntroTwisted} (i) and (ii) in which $\lambda_1 = \dots = \lambda_r= e$.
\end{proof}


\end{document}